\documentclass[a4paper,12pt]{article}

\usepackage{amsmath,amsthm}
\usepackage{amssymb,amsfonts}
\usepackage[french,english]{babel}
\usepackage[latin1]{inputenc}
\usepackage{color}
\usepackage{mathrsfs}
\usepackage{cancel}

\newtheorem{theorem}{Theorem}[section]
\newtheorem{proposition}[theorem]{Proposition}
\newtheorem{corollary}[theorem]{Corollary}

\theoremstyle{definition}
\newtheorem{remark}[theorem]{Remark}
\newtheorem{definition}[theorem]{Definition}

\def\[#1\]{\begin{align*}#1\end{align*}}

\newcommand{\hM}{\hat{M}}
\newcommand{\hn}{\hat{n}}
\newcommand{\hg}{\hat{g}}
\newcommand{\hx}{\hat{x}}

\newcommand{\hX}{\hat{X}}
\newcommand{\hY}{\hat{Y}}
\newcommand{\hZ}{\hat{Z}}
\newcommand{\hR}{\hat{R}}
\newcommand{\hh}{\hat{h}}
\newcommand{\hH}{\hat{H}}
\newcommand{\hf}{\hat{f}}
\newcommand{\hF}{\hat{F}}
\newcommand{\hU}{\hat{U}}

\newcommand{\hN}{\hat{N}}

\newcommand{\qmhm}{Q(M,\hat{M})}
\newcommand{\tsmthm}{T^* M \otimes T\hat{M}}
\newcommand{\tsxmthxm}{T^*_{x} M \otimes T_{\hx}\hat{M}}
\newcommand{\dns}{\mathscr{D}_{NS}}
\newcommand{\odns}{\mathcal{O}_{\mathscr{D}_{NS}}}
\newcommand{\dr}{\mathscr{D}_{R}}
\newcommand{\odr}{\mathcal{O}_{\mathscr{D}_{R}}}
\newcommand{\lns}{\mathscr{L}_{NS}}
\newcommand{\lr}{\mathscr{L}_{R}}
\newcommand{\pq}{\pi_Q}
\newcommand{\pqm}{\pi_{Q,M}}

\newcommand{\rarrow}{\rightarrow}
\newcommand{\tbar}{\overline{T}}
\newcommand{\nablabar}{\overline{\nabla}}
\newcommand{\hnabla}{\hat{\nabla}}
\newcommand{\hgamma}{\hat{\gamma}}
\newcommand{\Xbar}{\overline{X}}

\newcommand{\sbar}{\overline{S}}
\newcommand{\hT}{\hat{T}}
\newcommand{\hS}{\hat{S}}
\newcommand{\Oh}{\mathcal{O}}
\newcommand{\onq}{|_{q}}
\newcommand{\onqz}{|_{q_0}}

\newcommand{\onxz}{|_{x_0}}

\newcommand{\onhxz}{|_{\hat{x}_0}}
\newcommand{\q}{q=(x,\hat{x};A)}
\newcommand{\qz}{q_0=(x_0,\hat{x}_0;A_0)}
\newcommand{\qo}{q_1=(x_1,\hat{x}_1;A_1)}

\newcommand{\Rol}{\mathsf{Rol}}

\newcommand{\IM}{\mathrm{im}}
\newcommand{\id}{\mathrm{id}}
\newcommand{\VF}{\mathrm{VF}}
\newcommand{\R}{\mathbb{R}}
\newcommand{\mc}{\mathcal}
\newcommand{\SO}{\mathrm{SO}}
\newcommand{\qmatrix}[1]{\left(\begin{matrix}#1\end{matrix}\right)}
\begin{document}

\title{Rolling Manifolds of Different Dimensions}

\author{Amina MORTADA\thanks{amina\underline{\;}mortada2010@hotmail.com, Universit\'e Paris-Sud 11, CNRS and STITS, Gif-sur-Yvette, 91192, France and Universit\'e Libanaise, EDST, Hadath, Liban} \and Petri KOKKONEN\thanks{petri.kokkonen@varian.com, Varian Medical Systems Finland} \and Yacine CHITOUR\thanks{yacine.chitour@lss.supelec.fr, Universit\'e Paris-Sud 11, CNRS and Sup\'elec, Gif-sur-Yvette, 91192, France}
}

\date{\today}

\maketitle

\begin{abstract}
If  $(M,g)$ and $(\hM,\hg)$ are two smooth connected complete oriented Riemannian manifolds of dimensions $n$ and $\hn$ respectively,
we model the rolling of $(M,g)$ onto $(\hM,\hg)$ as a driftless control affine systems describing two possible constraints of motion: the first rolling motion $\Sigma_{NS}$ captures the no-spinning condition only and the second rolling motion $\Sigma_{R}$ corresponds to rolling without spinning nor slipping. Two distributions of dimensions  $(n + \hn)$ and $n$, respectively, are then associated to the rolling motions $\Sigma_{NS}$ and $\Sigma_{R}$ respectively. This generalizes the rolling problems considered in \cite{ChitourKokkonen1} where both manifolds had the same dimension.
The controllability issue is then addressed for  both $\Sigma_{NS}$ and $\Sigma_{R}$ and completely solved for  $\Sigma_{NS}$.  As regards to $\Sigma_{R}$, basic properties for the reachable sets are provided as well as the complete study of the case $(n,\hn)=(3,2)$ and some sufficient conditions for non-controllability.
\end{abstract}

\tableofcontents

\section{Introduction}
In this paper, we continue the study initiated in \cite{ChitourKokkonen1}

of the rolling of two smooth connected complete oriented Riemannian manifolds $(M,g)$ and $(\hM,\hg)$ of dimensions $n$ and $\hn$ respectively, where the integers $n$ and $\hn$ are now not necessarily equal. Two sets of constraints are usually considered, namely the rolling without spinning on the one hand and the rolling without spinning nor spinning on the other hand.  Most of the existing work on the subject concerns the case where both manifolds  have the same dimension, i.e., $n=\hn$, cf. \cite{agrachev99,bryant-hsu,
CGK13:A,CGK13:B,ChitourKokkonen,kokkonen12:A,kokkonen12:B,norway,
marigo-bicchi,sharpe97} and references therein. In particular, the state space $Q$ for both types of rolling models is a fiber bundle over $M\times \hM$ where the typical fiber
consists of all the \emph{partial isometries}  $A : T_x M \rarrow T_{\hx} \hM$ (orientation preserving, if applicable).

The kinematic constraints of rolling without spinning simply says geometrically that parallel vector fields along a curves on $M$ are mapped to parallel vector fields along curves on $\hM$. Thus the smooth no-spinning distribution $\dns$ on $Q$ is defined by making use of the derivative of the parallel transport of the pair of vector fields $(X,\hat{X})$ along curves on $(M,\hM)$.

Then the distribution $\dr$ satisfying additionally  the kinematic constraint of no-slipping is obtained as a sub-distribution of $\dns$ by imposing $\hat{X}=AX$. Two driftless control systems $\Sigma_{NS}$ and $\Sigma_{R}$ are therefore defined on $Q$,  associated respectively with $\dns$ and to $\dr$, (see \cite{agrachev04,jurd} for more details on control systems).

The main issue consists in addressing the problem of complete controllability of $\Sigma_{NS}$ and $\Sigma_{R}$ in terms of the geometries of $M$ and $\hM$. That means to provide necessary and/or sufficient conditions so that, for every pair of points $(q_{init},q_{final})$ in $Q$, there exists a curve $\gamma$ steering $q_{init}$ to $q_{final}$ and tangent to $\dns$ (respectively to $\dr$). The attention is then focused on the details of the ``rolling'' orbits, which are the reachable sets associated to the distributions $\dns$ and $\dr$. One is able to provide a complete answer for $\Sigma_{NS}$ (cf. \cite{CGK13:A,CGK13:B}) by directly describing the corresponding rolling orbits in terms of the holonomy groups of $(M,g)$ and $(\hM,\hg)$. As for $\Sigma_{R}$, the situation is far more complicated. The standard strategy consists in computing the iterated Lie brackets of sections of $\dr$ and in verifying whether they span the tangent space of $Q$ at each point. By calculating the first Lie bracket, one comes to the definition of a tensor $\Rol$ called the {\it rolling curvature}, cf. \cite{ChitourKokkonen1}, which can be seen as the difference between the curvature tensor of $M$ and that of $\hM$. Hence, in the two dimensional case, the rolling curvature essentially reduces to the difference between Gaussian curvatures. The higher order Lie brackets involve linear combinations of (higher order) covariant derivatives of the rolling curvature $\Rol$ and it seems impossible to calculate in general the dimension of the evaluation at each point of the Lie algebra of these iterated Lie brackets. However, satisfactory answers for complete controllability of $\Sigma_{R}$ were given in the case where both manifolds are three dimensional (cf. \cite{ChitourKokkonen1}) or if one of them is of constant curvature (cf. Section 6 in \cite{ChitourKokkonen}). Moreover, one is even able in the two dimensional case to provide motion algorithms for $\Sigma_R$, cf. \cite{ACL,chelouah01,mar-bic2}.

In the present paper, we extend the constructions and basic results of \cite{ChitourKokkonen1} to the case where $(M,g)$ and $(\hM,\hg)$ do not have necessarily the same dimension, i.e., $n$ is not necessarily equal to $\hn$. The first modification consists in generalizing the definition of the state space $Q$ to the following\\
$
\begin{array}{clll}
(i) & \text{if } n \leq \hn, & \qmhm:= &\!\!\! \{A \in \tsmthm \mid \hg(AX,AY) = g (X,Y), \; X,Y \in T_{x} M, x \in M \},\\
\\
(ii) & \text{if } n \geq \hn, & \qmhm:= &\!\!\! \{A \in \tsmthm \mid \hg(AX,AY) = g (X,Y), \; X,Y \in (\ker A)^{\perp},\\
 & & & \hspace{3.2cm} \text{ $A$  is  onto the  tangent  space  of } \hM \},
\end{array}
$
\\
(in other words, $Q(M,\hat{M})$ is the set of \emph{partial isometries} of maximal rank),
and in defining rigorously  the distributions $\dns$, $\dr$ as well as the rolling curvature $\Rol$. We then provide basic properties of the rolling orbits associated to $\dns$ and $\dr$ respectively. In the case where $n\neq \hn$,
some non controllability results will be given, namely in the presence of totally geodesic submanifolds in $\hM$ as well as results in the case $\vert n-\hn\vert=1$ (Proposition \ref{p15} and Proposition \ref{p16} below). Finally, we will completely solve the issue of complete controllability associated to $(\Sigma)_{NS}$ in the spirit of what has been done in \cite{ChitourKokkonen},
and we will study the case where $(n,\hn)$=$(3,2)$ by using results of \cite{ChitourKokkonen1}.
Parts of the results of this paper have already appeared in the preprint \cite{ChitourKokkonen}.\\

\textbf{Acknowledgements.} The first author would like to thank the Lebanese National Council for Scientific Research (CNRS) and Lebanese University for their financial support to this work.

\section{Notations}\label{sec:notations}
Unless otherwise stated, all manifolds considered in this paper are finite dimensional, smooth and connected. If  in addition a manifold $M$ is endowed a Riemannian metric $g$, then $(M,g)$ is assumed to be complete and oriented and we use $\Vert v\Vert_g$ to denote $g(v,v)^{1/2}$ for every $v\in T_xM$, where $x$ is an arbitrary point of $M$.
Let $L: V \rarrow W$ be a $\mathbb{R}$-linear map where $V$ and $W$ are two $\mathbb{R}$-linear spaces with dimensions $n$ and $n'$ respectively. Taking $F = (v_{i})_{i=1}^{n}$ and $G = (w_{i})_{i=1}^{n'}$ two bases of $V$ and $W$ respectively, the $(n' \times n)$- real matrix of L w.r.t. $F$ and $G$ is denoted by $\mathcal{M}_{F,G} (L)$ and given by $L (v_{i}) = \sum_{j} \mathcal{M}_{F,G} (L)_{i}^{j} w_{j}$.
Furthermore, $\tsxmthxm$ is canonically identified with the linear space of the $\mathbb{R}$-linear map $A : T_x M \rarrow T_{\hx} \hM$.

If $E,M,F$ are manifolds, a smooth bundle $\pi_{E,M} : E \rarrow M$ is a smooth map such that  for every $x \in M$ there exists a neighbourhood $U$ of $x$ in $M$ and a smooth diffeomorphism $\tau : \pi_{E,M}^{-1} (U) \rarrow U \times F$ so that $pr_1 \circ \tau = (\pi_{E,M} |_{\pi(U)})^{-1}$, where $pr_1$ stands for the projection onto the first factor. Then, $F$ is called the typical fiber of $\pi_{E,M}$ and $\tau$ is a (smooth) local trivialization of $\pi_{E,M}$. In the case where $F$ is a finite dimensional $\mathbb{R}$-linear space, we get a (smooth) vector bundle.

The set $E |_{x} = \pi_{E,M}^{-1} (x) := \pi_{E,M}^{-1} ({x}) $ is called the $\pi_{E,M}$-fiber over $x$. A smooth section of a bundle $\pi_{E,M}$ is a smooth map $s: M \rarrow E$ that satisfies $\pi_{E,M} \circ s = \id_{M}$. When the context is clear, we simply write $\pi$ for $\pi_{E,M}$.

A distribution $\mathscr{D}$ over a manifold $M$ is a smooth assignment $x\mapsto \mathscr{D} |_x$ where $\mathscr{D} |_x\subset T_xM$. An absolutely continuous curve $c : I \rarrow M$ defined on an interval $I \subset \mathbb{R}$ is $\mathscr{D}$-admissible curve if it is tangent to $\mathscr{D}$ almost everywhere (a.e.), i.e., for almost every $t \in I$, $\dot{c} (t) \in \mathscr{D} |_{c(t)}$. For $x_0 \in M$, the subset $\Oh_{\mathscr{D}} (x_0)$ of $M$ formed by the endpoints of all $\mathscr{D}$-admissible curves of $M$ starting at $x_0$ is called the $\mathscr{D}$-orbit through $x_0$. By the Orbit Theorem (see \cite{AgrachevSachkov}), it follows that $\Oh_{\mathscr{D}} (x_0)$ is an immersed smooth submanifold of $M$ containing $x_0$,
and that one can restrict the class of curves defining the orbit to the piecewise smooth ones.
We call a distribution $\mathscr{D}'$ on $M$ a subdistribution of $\mathscr{D}$ if $\mathscr{D}' \subset \mathscr{D}$. An immediate consequence of the definition of the orbit shows that $\Oh_{\mathscr{D}'} (x_0) \subset \Oh_{\mathscr{D}} (x_0)$ for all $x_0 \in M$.

For a smooth map $\pi: E \rarrow M$ and $y \in E$, let $V |_{y} (\pi)$ be the set of all $Y \in T |_{y} E$ such that $\pi_{*} (Y) = 0$. If $\pi$ is a bundle then the collection of spaces $V |_{y} (\pi)$, $y \in E$, defines a smooth \emph{vertical} distribution $V (\pi)$ on $E$.

When $\pi: E \rarrow M$ and $\eta: F \rarrow M$ are vector bundles over a manifold $M$, let $C^{\infty} (\pi ,\eta)$ be the set of smooth maps $g : E \rarrow F$ such that $\eta \circ g = \pi$.
Given $f \in C^{\infty} (\pi ,\eta)$ and $u, w \in \pi^{-1} (x)$, the vertical derivative of $f$ at $u$ in the direction $w$ is defined as
\[
\nu(w) |_{u} (f) := \frac{d}{dt} |_0 f(u+tw) \in \nu^{-1}(x).
\]

A smooth map $f: M \rarrow \hM$ is a local isometry between two Riemannian manifolds $(M,g)$ and $(\hM , \hg)$ if, for all $x \in M$, $f_{*} |_{x}: T_{x} M \rarrow T_{f(x)} \hM$ is an isometric linear map.
If moreover $f$ is bijective, it is called an isometry, and $(M,g)$, $(\hM, \hg)$ are said to be isometric.

We use $Iso (M,g)$ to denote the smooth Lie group of isometries of $(M,g)$.\\
We use $(\overline{M}, \overline{g})$ to denote  $(M,g) \times (\hM, \hg) $,  the Riemannian product manifold of $M$ and $\hM$, endowed with the product metric $\overline{g} := g \oplus \hg$.
Similarly, $\nabla$, $\hnabla$, $\nablabar$ (resp. $R$, $\hR$, $\overline{R}$) represent the Levi-Civita connections (resp. the Riemannian curvature tensors) of $(M,g)$, $(\hM, \hg)$, $(\overline{M}, \overline{g})$, respectively.

If $\gamma: I \rarrow M$ is an absolutely continuous curve defined on real interval $I \ni 0$ and  $T_0$ is any tensor at $\gamma(0)$ , we use $(P^{\nabla^{g}})_0^{t} (\gamma) T_0$ to denote the parallel transport of $T_0$ along $\gamma$ from $\gamma(0)$ to $\gamma(t)$ w.r.t. the Levi-Civita connection of $(M,g)$.\\
Furthermore, if $s,t\in I$ and $F \in Iso(M,g)$, then one has (see \cite{Sakai}, page 41, Eq. (3.5)) that
\begin{equation}\label{DS}
F_{*} |_{\gamma (t)} \circ (P^{\nabla^{g}})_{s}^{t} (\gamma) = (P^{\nabla^{g}})_{s}^{t} (F \circ \gamma) \circ F_{*} |_{\gamma (s)}.
\end{equation}
For every point $x$ of  a Riemannian manifold $M$ and $k \in \mathbb{N}$, let $\mathcal{L}_{k} (M) |_{x}$ be the space of linear maps $\mathbb{R}^{k} \rarrow T_{x} M$ and set $\mathcal{L}_{k} (M) := \bigcup_{x \in M} \mathcal{L}_{k} (M) |_{x}$.
\begin{definition}\label{d5}
We define the subset $O_{k} (M)$ of $\mathcal{L}_{k} (M)$ of all elements $B \in \mathcal{L}_{k} (M) |_{x}$, $x \in M$, such that
\begin{enumerate}
\item[(i)]
if $1 \leq k \leq \dim M$, $\Vert B u \Vert_{g} = \Vert u \Vert_{\mathbb{R}^{k}}$ for all $u \in \mathbb{R}^{k}$;
\item[(ii)]
if $k \geq \dim M$, then $B$ is surjective and $\Vert B u \Vert_{g} = \Vert u \Vert_{\mathbb{R}^{k}}$ for all $u \in (\ker B)^{\perp}$ (where $\perp$ is taken with respect to the Euclidean inner product in $\mathbb{R}^{k}$).
\end{enumerate}
In other terms, $O_{k} (M)$ is composed of the partial isometries $\mathbb{R}^{k} \rarrow T_{x} M$ of maximal rank.
\end{definition}

Consider the map $\pi_{\mathcal{L}_{k} (M) } : \mathcal{L}_{k} (M) \rarrow M$ defined by $B \mapsto x$ for $B \in \mathcal{L}_{k} (M) |_{x}$ and $\pi_{O_{k} (M) } := \pi_{\mathcal{L}_{k} (M) } |_{O_{k} (M) }: O_{k} (M) \rarrow M$. We have the following standard results.

\begin{proposition}\label{p22}
For every $k\in \mathbb{N}$, $\pi_{\mathcal{L}_{k} (M) }$ is a vector bundle over $M$, isomorphic to $\bigoplus_{i=1}^{k} T M \rarrow M$; $\pi_{O_{k} (M) }$ defines a subbundle of $\pi_{\mathcal{L}_{k} (M) }$ whose typical fiber is $O_{k} (\mathbb{R}^{n})$, where $\mathbb{R}^{n}$ is equipped with the Euclidean metric.
Moreover,
$O_{k} (M) $ is connected for any $k \neq n$ and $O_{n} (M)$ is connected whenever $M$ is not orientable.
\end{proposition}
If $n,\hn$ are two positive integers, let $(\mathbb{R}^{n})^{*} \otimes \mathbb{R}^{\hat{n}}$ is the set of $\hn \times n$ real matrices. We denote
\begin{eqnarray}\label{B}
SO(n,\hn) :=
\left\lbrace
\begin{array}{lr}
\{ A \in (\mathbb{R}^{n})^{*} \otimes \mathbb{R}^{\hn} \mid A^{T} A = \id_{\mathbb{R}^{n}} \}, & if \; n < \hn, \\
\{ A \in (\mathbb{R}^{n})^{*} \otimes \mathbb{R}^{\hn} \mid A A^{T} = \id_{\mathbb{R}^{\hn}} \}, & if \; n > \hn, \\
SO(n), & if \; n = \hn,
\end{array}
\right.
\end{eqnarray}
where $A^{T}$ is used to denote the usual transpose with respect to the inner product of the appropriate vector space.
We give the following matrix form $I_{n,\hn} \in SO(n,\hn) $,
\begin{eqnarray}\label{DR}
I_{n,\hn} =
\left\lbrace
\begin{array}{cc}
\left(
\begin{array}{c}
\id_{\mathbb{R}^{n}}\\
0
\end{array}
\right), & \quad if \; n \leq \hn,\\
\\
\left(

\id_{\mathbb{R}^{\hn}} \quad 0

\right), & \quad if \; n \geq \hn.
\end{array}
\right.
\end{eqnarray}
If  $(M,g)$ is an $n$-dimensional Riemannian manifold and $x \in M$, the identification of the tangent space $T_{x} M$ with the Euclidean space $\mathbb{R}^{n}$ allows one to write $SO(T_{x} M) = SO(n)$ and $\mathfrak{so} (T_{x} M) = \mathfrak{so} (n)$. We also denote $\mathfrak{so} (M) := \bigcup_{x \in M} \mathfrak{so} (T_{x} M)$ as the set $\{B \in T^* M\otimes TM \mid B^T + B =0 \}$.

Let $N$ be a manifold. A loop $\gamma:[a,b] \rarrow N$  based at $y \in N$  is a curve verifying $\gamma (a) = \gamma (b) = y$ and let $\Omega_{y} (N)$ be the set of all piecewise $C^1$-loops $[0,1] \rarrow N$ of $N$ based at $y$. On the other hand, if $(N,h)$ is a Riemannian manifold, then the holonomy group $H^{\nabla^{h}} |_{y}$ of $N$ at $y$ is defined by
$$
H^{\nabla^{h}} |_{y} = \{ (P^{\nabla})_0^1 (\gamma) \mid \gamma \in \Omega_{y} (N)\},
$$
and it is a subgroup of $O (T_{y} N)$ made of all $h$-orthogonal transformations of $T_{y} N$. If $N$ is oriented, then $H^{\nabla^{h}} |_{y}$ is a subgroup of $SO(T_{y} N)$. If $n = \dim N$ and $F$ is an orthonormal frame of $N$ at $y$, we write
$$
H^{\nabla^{h}} |_{F} = \{ \mathcal{M}_{F,F} (A) \mid A \in H^{\nabla^{h}} |_{y} \}.
$$
This is a subgroup of $SO(n)$, isomorphic (as Lie group) to $ H^{\nabla^{h}} |_{y}$. The Lie algebra of the holonomy group $ H^{\nabla^{h}} |_{y}$ (resp. $H^{\nabla^{h}} |_{F}$) will be denoted by $\mathfrak{h}^{\nabla^{h}} |_{y}$ (resp. $\mathfrak{h}^{\nabla^{h}} |_{F}$). Then $\mathfrak{h}^{\nabla^{h}} |_{y}$ is a Lie subalgebra of the Lie algebra $\mathfrak{so} (T_{y} N)$ of $h$-antisymmetric linear maps $ T_{y} N \rarrow T_{y} N$ while $\mathfrak{h}^{\nabla^{h}} |_{F}$ is a Lie subalgebra of $\mathfrak{so} (n)$.

\section{Rolling Motions}

\subsection{The State Space $Q$}

\begin{definition}
Let $(M,g)$ and $(\hM,\hg)$ be two Riemannian manifolds of dimensions $n$ and $\hn$ respectively. The state space $Q=\qmhm$ for the problem of rolling of $M$ against $\hM$ considered below is defined as follows:
\begin{enumerate}
\item[(i)]
if $n \leq \hn$,
$$
\qmhm:=\{A \in \tsmthm \mid \hg(AX,AY) = g (X,Y), \; X,Y \in T_{x} M, x \in M \}.
$$
\item[(ii)]
if $n \geq \hn$,
$$
\begin{array}{lc}
\qmhm:= & \{A \in \tsmthm \mid \hg(AX,AY) = g (X,Y), \; X,Y \in (\ker A)^{\perp}, \\
& \text{ A  is  onto  a  tangent  space  of } \hM \}.
\end{array}
$$
\end{enumerate}
\end{definition}
Writing $A^{T}: T_{\hx} \hM \rarrow T_{x} M$ the $(g, \hg)$-transpose of $A$, we have that $(\ker A)^{\perp} = \IM(A^{T})$,
and evidently $A^{T} A = \id_{T_{x} M}$ if $n \leq \hn$ and $A A^{T} = \id_{T_{\hx} \hM}$ if $n \geq \hn$. Also, define
\begin{eqnarray}
\begin{array}{c}
\pi_{\qmhm, M} : = \pi_{\tsmthm, M} |_{\qmhm} : \qmhm \rarrow M,\\
\pi_{\qmhm, \hM} : = \pi_{ \tsmthm, \hM} |_{\qmhm} : \qmhm \rarrow \hM.
\end{array}
\end{eqnarray}
If $q\in \qmhm$, we use the notation $q=(x,\hx ;A)$ where $x=\pi_{\qmhm,M}(q)$ and
$\hx=\pi_{\qmhm,\hM}(q)$.

\begin{proposition}\label{p10}
\begin{enumerate}
\item[$(i)$]
The space $\qmhm$ is a smooth closed submanifold of $\tsmthm$ of dimension:
$$
dim (Q) = n + \hat{n} + n \hn - \frac{N (N +1)}{2}, \; where \; N := min \{n,\hn \},
$$
and $\pi_{\qmhm, M}$ is a smooth subbundle of $\pi_{\tsmthm, M}$ with typical fiber $O_{n} (\hM)$.
\item[$(ii)$] The map
$$
\tau_{M,\hM} : \tsmthm \rarrow T^{*} \hM \otimes T M; \; \; \; (x,\hx; A) \mapsto (\hx,x ; A^{T}),
$$
is a diffeomorphism and its restriction to $\qmhm$gives the diffeomorphism

\begin{equation}\label{DT}
\tbar: \qmhm  \rarrow  Q(\hM,M)=:\hat{Q}; \;\;\; \tbar(x,\hx ;A) = \tau_{M,\hM} |_{Q} (x,\hx ;A) = (\hx,x ;A^{T}).
\end{equation}

\item[$(iii)$]
If $n \neq \hn$ or if one of $M$ and $\hM$ is not orientable, then the space $\qmhm$ is connected.
\end{enumerate}
\end{proposition}
\begin{proof}
($i$) It is clearly enough to prove the result only for $n\leq \hn$. In that case, the (vertical) fiber of $Q$ is isomorphic to the grasmannian of $n$-dimensional planes in an $\hn$-dimensional euclidean space, hence the result.
\end{proof}

\begin{enumerate}
\item[$(ii)$]
First at all, we see that $\tau_{\hM,M}$ is the inverse map of $\tau_{M,\hM}$, thus $\tau_{M,\hM}$ is a diffeomorphism. Moreover, one has $(\hx,x ;A^{T}) \in \hat{Q}$ for every $(x,\hx ;A) \in Q$. Indeed, according to (i), one may assume that $n \leq \hn$. Let $\hX,\hY \in (\ker A^{T})^{\perp}$. Since $(\ker A^{T})^{\perp} = \IM (A)$, there are $X,Y \in T_{x} M$ such that $AX = \hX , AY=\hY$,
and because $A^{T} A = \id_{T_{x}M}$, we get that $g(A^{T} \hX , A^{T} \hY) = g(X,Y) = \hg (AX,AY)= \hg(\hX,\hY)$. Now, take the map
$$
\overline{S}: \hat{Q} \rarrow Q; \; \; \; \overline{S}(\hx,x ;B) = (x,\hx;B^{T}),
$$
which is well-defined because $(\ker B)^{\perp} = \IM (B^{T})$ and $B B^{T}= \id_{T_{x}M}$. Thus, for all $X,Y \in T_{x} M$, one obtains $\hg (B^{T} X,B^{T} Y) = g (B B^{T} X, B B^{T} Y) = g (X,Y)$. Therefore, $\tbar$ and $\overline{S}$ are smooth inverse maps to each other.
\item[$(iii)$] This follows from Proposition \ref{p22}.

\hfill$\Box$\medskip
\end{enumerate}
\begin{corollary}\label{cor-not}
The map $\pi_{\qmhm}: \qmhm \rarrow M \times \hM$ is a bundle whose typical fiber is diffeomorphic to $O_{n} (\mathbb{R}^{\hn})$.
\end{corollary}

\proof
For a given point $(x_0, \hx_0) \in M \times \hM$, take any $g$-orthonormal (resp. $\hg$-orthonormal) frame $F = (X_1, ..., X_{n})$ (resp. $\hF = (\hX_1, ..., \hX_{\hn})$) defined on some open neighbourhood $U$ of $x_0$ (resp. $\hU$ of $\hx_0$). Fix a $\q \in (\pi_{\qmhm})^{-1} (U \times \hU)$, define $G_{F, \hF} (A)$ to be the $\hn \times n$-matrix whose element on the $i$-th row, $j$-th column is $\hg (\hX_{i} |_{\hx} , A X_{j} |_{x})$ and set
\begin{align*}
\tau_{F,\hF} : (\pi_{\qmhm})^{-1} (U \times \hU) \rarrow (U \times \hU) \times (\mathbb{R}^{n})^{*} \otimes \mathbb{R}^{\hn}; \;\;\; \tau_{F,\hF} (x,\hx;A) = ((x,\hx), G_{F, \hF} (A)).
\end{align*}
Using Proposition \ref{p10}, it is easy to see that $\tau_{F,\hF}$ is smooth, injective and its image is $(U \times \hU) \times O_{n} (\mathbb{R}^{\hn})$. Moreover, its inverse map $\tau_{F,\hF}^{-1}: (U \times \hU) \times O_{n} (\mathbb{R}^{\hn}) \rarrow (\pi_{\qmhm})^{-1} (U \times \hU)$ is given by
\begin{align*}
\tau_{F,\hF}^{-1} ((x,\hx),B) = (x,\hx; \sum_{j=1}^{n} \sum_{i=1}^{\hn} B_{ij} g (\cdot, X_{j} |_{x}) \hX_{i}),
\end{align*}
where $B_{ij}$ is the element on $i$-th row, $j$-th column of $B$.
The fact that $\tau_{F,\hF}$ and $\tau_{F,\hF}^{-1}$ are smooth is easily established.
\hfill$\Box$\medskip

\begin{proposition}\label{p9}
Let $q=(x,\hx;A) \in Q$ and $B \in \tsxmthxm$. Then $\nu (B)|_{q}$ is tangent to $Q$ $($i.e. is an element of $V |_{q} (\pq))$ if and only if
\begin{enumerate}
\item[(i)]
$A^{T} B \in \mathfrak{so} (T_{x} M)$, if $n \leq \hn$.
\item[(ii)]
$BA^{T} \in \mathfrak{so} (T_{\hx} \hM)$, if $n \geq \hn$.
\end{enumerate}
\end{proposition}

\begin{proof}

Note that the set of $B \in \tsxmthxm$ such that $A^{T} B \in \mathfrak{so} (T_{x} M)$ and the set of $B \in \tsxmthxm$ such that $B A^{T}  \in \mathfrak{so} (T_{\hx} \hM)$ both have dimension equal to dim $\pq^{-1} (x,\hx)$. Therefore, it is sufficient to show that $V |_{q} (\pq) \subseteq \mathfrak{so} (T_{x} M)$ when $n \leq \hn$ and $V |_{q} (\pq) \subseteq \mathfrak{so} (T_{\hx} \hM)$ when $n \geq \hn$. We only prove Item $(i)$ since the other follows by using Eq. \eqref{DT}.
If $n \leq \hn$ and $X \in T_{x} M$, then $A^{T} A X = X$. For any $B \in \tsxmthxm$ tangent to $Q$, we have $\nu(B) |_{q} X =0$. Then, $0 = \nu(B) |_{q} (.)^{T} (.) X = B^{T} A X + A^{T} B X$ and hence $B^{T} A  + A^{T} B =0$ because $X$ was arbitrary.
Same analysis as (i): if $n \geq \hn$ and $\hX \in T_{\hx} \hM$, then we have $ A A^{T} \hX = \hX$. For any $B \in \tsxmthxm$ tangent to $Q$, we have $\nu(B) |_{q} \hX =0$. Then, $0 = \nu(B) |_{q} (.) (.)^{T} \hX = B A^{T} \hX + A B^{T} \hX$ and hence the conclusion.
\end{proof}

\subsection{The Rolling Lifts and Distributions}

Since we are interested in the rolling motion without spinning nor slipping, we formulate these conditions by taking an absolutely continuous curve on $Q$, $q: [a,b] \rarrow Q$; $t \mapsto (\gamma(t), \hgamma(t); A(t))$ and making the following definitions.

\begin{definition}
The curve $q (\cdot) $ is said to describe:
\begin{description}
\item[$(i)$] A rolling motion without spinning of $M$ against $\hM$ if:
\begin{equation}\label{CM}
\nablabar_{(\dot{\gamma} (t), \dot{\hgamma} (t))} A (t) =0 \text{ for a.e. } t \in [a,b].
\end{equation}
\item[$(ii)$]  A rolling motion without slipping of $M$ against $\hM$ if we have:
\begin{equation}\label{CN}
A(t) \dot{\gamma} (t) = \dot{\hgamma} (t) \text{ for a.e. } t \in [a,b].
\end{equation}
\item[$(iii)$]  A rolling motion without slipping nor spinning of $M$ against $\hM$ if both conditions $(i)$ and $(ii)$ hold true.
\end{description}
By Item $(iii)$ above, we get that the curves $q$ of $Q$ describing the rolling motion without slipping and spinning of $M$ against $\hM$  are exactly the integral curves of  the following driftless control affine system
\begin{eqnarray} \label{A}
\Sigma_{R}:\quad 
\left\lbrace
\begin{array}{lll}
\dot{\gamma} (t) & = & u(t), \\
\dot{\hgamma} (t) & = & A(t) u(t), \; \text{ \; \; \; for a.e. } t \in [a,b],\\
\nablabar_{( u (t), A (t) u(t))} A (t)& = &  0,\\
\end{array}
\right.
\end{eqnarray}
where the control $u$ is a measurable $TM$-valued function defined on some finite interval $I\subset \mathbb{R}$. (In Appendix \ref{app0}, we provide an expression in (local) coordinates of 
$(\Sigma)_{R}$ as well as the control system describing the rolling motion without spinning only of $M$ against $\hM$.) 
\end{definition}

\begin{proposition}\label{p1}
Let $A_0$ be a (1,1)-tensor on $M\times \hM$ (i.e. $\in {T_{1}^{1}}_{\; (x_0, \hx_0)} (M \times \hM) $ for $(x_0, \hx_0) \in M \times \hM$) and
$t \mapsto q(t)=(\gamma(t), \hgamma (t); A(t))$ be an absolutely continuous curve in $T^*M\otimes T\hat{M}$
defined on some real interval $I \ni 0$ and satisfying \eqref{CM}.
Then we have, for all $t\in I$,
\[
& A(t) = P_0^{t} (\hgamma) \circ A(0) \circ P_{t}^0 (\gamma),\\
& A(0) \in Q \ \Longrightarrow \ A(t)\in Q.
\]
\end{proposition}

\begin{proof}
For the first implication,
define $B(t):=P_0^t(\hgamma)\circ A(0)\circ P_t^0(\gamma)$.
Evidently $B(0)=A(0)$,
and if $X(t)$ is an arbitrary vector field along $\gamma(t)$,
we have that $B(t)X(t)$ is a vector field along $\hgamma(t)$, and
\[
&\big(\overline{\nabla}_{(\dot{\gamma}(t),\dot{\hgamma}(t))}B(t)\big)X(t)
+B(t)\nabla_{\dot{\gamma}(t)} X(t)
=\hnabla_{\dot{\hgamma}(t)}(B(t)X(t))
=\hnabla_{\dot{\hgamma}(t)}\Big(P_0^t(\hgamma)\big(A(0)P_t^0(\gamma)X(t)\big)\Big) \\
=&P_0^t(\hgamma)\frac{d}{dt}\big(A(0)P_t^0(\gamma)X(t)\big)
=(P_0^t(\hgamma)\circ A(0)\circ P_t^0(\gamma))\big(\nabla_{\dot{\gamma}(t)}X(t)\big)
=B(t)\nabla_{\dot{\gamma}(t)}X(t),
\]
which, since $X(t)$ was arbitrary, would mean that
$\overline{\nabla}_{(\dot{\gamma}(t),\dot{\hgamma}(t))}B(t)=0$.
By the basic uniqueness result for the first order ODEs,
we thus have $A(t)=B(t)$ for all $t\in I$.

For the second implication, let $Y \in T_{\gamma(0)} M $, $\hY \in T_{\hgamma(0)} \hM$ and set $Y(\cdot)$, $\hY(\cdot) $  the parallel transports of $Y,\hY$ along $\gamma (.)$ and $\hgamma(.)$ respectively. Next, suppose that $A(0) \in Q |_{(\gamma(0), \hgamma(0))}$ and denote $A(t) = P_0^{t} (\hgamma) \circ A(0) \circ P_{t}^0 (\gamma)$. Then $A(0) \in \tsmthm$, but from the first implication we obtain $A(t) \in  \tsmthm$ for all $t \in I$. So $A(t) Y(t) \in T_{\hgamma (t)} \hM$ and thus,
$$
\frac{d}{dt} \Vert A(t) Y (t) \Vert_{\hg}^2 = 2 \hg ((\nablabar_{(\dot{\gamma} (t) , \dot{\hgamma} (t))} A(.)) Y(t) + A (t) (\nabla_{\dot{\gamma} (t)} Y (.)) , A(t) Y(t)) =0.
$$
If $n \leq \hn$, the initial condition for the preceding term is $\Vert A(0) Y (0) \Vert_{\hg}^2 = \Vert A(0) Y \Vert_{\hg}^2 = \Vert Y \Vert_{g}^2$. On the other hand, $\frac{d}{dt} \Vert Y (t) \Vert_{g}^2 = 0$ and the initial condition is $\Vert Y (0) \Vert_{g}^2 = \Vert Y \Vert_{g}^2$. So, $\Vert A(t) Y (t) \Vert_{\hg}^2 = \Vert Y (t) \Vert_{g}^2$. Since the parallel transport $P_0^{t} (\gamma) : T_{\gamma(0)} M \rarrow T_{\gamma(t)} M$ is a linear isometric isomorphism for every $t$, this proves $ \hg (A(t) X , A(t) Y) = g (X,Y)$ for every $X$, $Y \in T_{\gamma(t)} M$. If $n \geq \hn$, we are able to repeat the previous method due to the fact $Y(t) \in (\ker A(t))^{\perp}$ if and only if $Y \in (\ker A(0))^{\perp}$.
\end{proof}

\begin{definition}\label{d1}
\begin{itemize}
\item[(i)]
Given $q=(x,\hx;A) \in \tsmthm$ and $X\in T_{x}M$, $\hX \in T_{\hx} \hM$,
one defines the no-spinning lift of $(X,\hat{X})$ to be the unique vector $\lns(X, \hat{X}) |_{q}$ of $\tsxmthxm$ at $q$
given by
\[
\lns(X, \hX)|_q = \frac{d}{dt} \big|_0 P_0^{t} (\hgamma) \circ A \circ P_{t}^0 (\gamma)  \quad \big(\in T_{q} (\tsmthm)\big),
\]
where $\gamma$ (resp. $\hat{\gamma}$) is any smooth curves on $M$ (resp. $\hM$)
such that $\gamma(0)=x$, $\dot{\gamma}(0)=X$ (resp. $\hat{\gamma}(0)=\hx$, $\dot{\hat{\gamma}}(0)=\hX$).

Moreover, if $X,\hat{X}$ are (locally defined) vector fields on $M,\hat{M}$, respectively, one writes $\lns(X,\hat{X})$
for the (locally defined) vector field on $\tsmthm$ whose value at $q$ is $\lns(X,\hat{X})|_q$.

\item[(ii)]
No-Spinning distribution $\dns$ on $\tsmthm$ is an $(n+\hn)$-dimensional smooth distribution, whose plane at $q=(x,\hx;A) \in \tsmthm$ is defined by
\[
\dns |_{q} = \lns (T_{(x,\hx)} (M \times \hM)) |_{q}.
\]
\end{itemize}
\end{definition}
By Proposition \ref{p1}, $\lns$ can be restricted to $Q$ so that
\[
\lns(X,\hat{X})|_q\in T_qQ,\quad \dns |_{q}\subset T_qQ,
\]
for any $q\in Q$ and $X\in T_x M$, $\hX\in T_{\hx} \hM$ as in the definition above.

Hence, we have $\dns |_{Q}$ is an $(n + \hn)$-dimensional (smooth) distribution on $Q$,
which we also write as $\dns$ in the sequel. The next proposition gathers basic properties of $\dns$.
\begin{proposition}\label{p2}
\begin{enumerate}
\item[$1.$]
$(\pi_{\tsmthm})_{*}$ (resp. $(\pq)_{*}$) maps $\dns |_{q}$ isomorphically onto $T_{(x,\hx)} (M \times \hM)$ for every $q =(x,\hx;A) \in \tsmthm$ (resp. $ \in Q$).
\item[$2.$]
If $\Xbar \in T_{(x,\hx)} (M \times \hM)$, $A$ is a local section of $\pi_{\tsmthm}$ and $A_{*}$ its push-forward, then we have:
\begin{equation}\label{D}
\lns (\Xbar) |_{A(x, \hx)} = A_{*} (\Xbar) - \nu (\nablabar_{\Xbar} A) |_{A(x, \hx)}.
\end{equation}
\item[$3.$]
An absolutely continuous curve $t \mapsto q(t)=(\gamma(t), \hgamma(t); A(t))$ on $\tsmthm$ or $Q$ 
is tangent to $\dns$ for a.e. $t$ if and only if $\nablabar_{(\dot{\gamma} (t), \dot{\hgamma} (t))} A =0$ for a.e. $t$.
\end{enumerate}
Recall that $\nablabar$ is the product (Levi-Civita) connection on $\overline{M}=M\times\hat{M}$.
\end{proposition}

\begin{proof}
The proofs of parts $1.$ and $2.$ follow that of Proposition 3.20 and Proposition 3.22 of Section 3 in \cite{ChitourKokkonen}.
Part $3.$ is a consequence of Eq. (\ref{D}) so that
\begin{equation*}
\lns (\dot{\gamma} (t) , \dot{\hgamma} (t)) |_{q(t)} = \dot{A} (t) - \nu (\nablabar_{(\dot{\gamma} (t) , \dot{\hgamma} (t))} A) |_{q(t)}.
\end{equation*}
\end{proof}

\begin{remark}\label{r1}
In the previous proposition, 
the two terms on the right side of Eq. (\ref{D})
are separately elements of $T_{q} (\tsmthm )$, but their difference belongs to $T_{q} Q$. Moreover, this equation indicates the decomposition of the map $A_{*}$ with respect to the two direct sum decompositions:
\[
T (\tsmthm) =& \dns \oplus_{\tsmthm} V(\pi_{\tsmthm}), \\
TQ =& \dns \oplus_{Q} V(\pq).
\]
\end{remark}

We shall now define a subdistribution $\dr$ of $\dns$ which has the property
that tangent curves to $\dr$ are exactly those curves in $Q$ (or $\tsmthm$) that verify both the no-slipping and no-spinning conditions,
i.e., are the curves modelled by the system $\Sigma_{(R)}$.

\begin{definition}\label{d2}
\begin{itemize}
\item[(i)]
For any $q=(x,\hx;A) \in \tsmthm$, the rolling lift of $X \in T_{x}M$
is the vector $\lr(X)|_q$ of $\tsmthm$ at $q$ defined by
\begin{align} \label{C}
\lr(X)|_q:= \lns (X, AX) |_{q}.
\end{align}
Moreover, if $X$ is a (locally defined) vector field on $M$, one writes $\lns(X)$
for the (locally defined) vector field on $\tsmthm$ whose value at $q$ is $\lns(X)|_q$.

\item[(ii)]
The Rolling distribution $\dr$ on $\tsmthm$ is the $n$-dimensional smooth distribution
whose plane at every $q=(x,\hx;A) \in \tsmthm$ is given by
\begin{align}\label{N}
\dr |_{q} := \lr (T_{x} M) |_{q}.
\end{align}
\end{itemize}
\end{definition}

Like right below the definition \ref{d1}, one can restrict $\lr$ to $Q$ such that
\[
\lr(X)|_q\in T_q Q,\quad \dr |_{q}\subset T_q Q,
\]
for all $q=(x,\hat{x};A)\in Q$ and $X\in T_xM$.

\begin{corollary}
\begin{itemize}
\item[(i)] $(\pqm)_{*}$ maps $\dr |_{q}$ isomorphically onto $T_x M$ for
for every $q=(x,\hat{x};A)\in \tsmthm$ (resp. $q\in Q$).

\item[(ii)] An absolutely continuous curve $t \mapsto q(t)=(\gamma(t), \hgamma(t); A(t))$ on $\tsmthm$ (resp. on $Q$)
is a rolling curve if and only if it is tangent to $\dr$ for a.e. $t$ i.e. if and only if $\dot{q} (t) = \lr (\dot{\gamma} (t)) |_{q(t)}$
for a.e. $t$.
\end{itemize}
\end{corollary}
While some of the results that follow hold true
in both spaces $Q$ and $\tsmthm$, we mainly focus on $Q$, which is the state space of primary interest
for the purposes of rolling. The generalization of such a result to $\tsmthm$, if it makes sense there,
is usually transparent, and, if need be, we will use such generalizations without further mention
for convenience in some of the forthcoming proof.

We have the following fundamental result whose proof follows the
same lines as that of Proposition 3.27 of Section 3 in \cite{ChitourKokkonen}.
\begin{proposition}\label{p3}
\begin{itemize}
\item[(i)] For every $q_0 = (x_0 , \hx_0 ; A_0) \in Q$ and every absolutely continuous $\gamma : [0,a] \rarrow M$, $a > 0$, such that $\gamma (0) = x_0$, there exists a unique absolutely continuous $q: [0,a'] \to Q$, $q (t) = ( \gamma (t) , \hat{\gamma} (t) ; A(t)) $, with $0 < a' \leq a$

which is tangent to $\dr$ a.e. and $q(0) = q_0$. We denote this unique curve $q$ by
\begin{equation}
t \mapsto q_{\dr} (\gamma, q_{0}) (t) = ( \gamma (t) , {\hat{\gamma}}_{\dr} (\gamma, q_0) (t) ; A_{\dr} (\gamma, q_0) (t)),
\end{equation}
and refer to it as the rolling curve with initial conditions $(\gamma, q_0)$, or along $\gamma$ with initial position $q_0$.

\item[(ii)] Moreover, if $(\hM,\hg)$ is a complete manifold, one can choose $a'=a$ above.
\item[(iii)] Conversely, any absolutely continuous curve $q: [0,a] \mapsto Q$ tangent to $\dr$ a.e. has the form $q_{\dr} (\gamma, q(0))$ where $\gamma = \pqm \circ q$.
\end{itemize}
\end{proposition}

\begin{remark}\label{r9}
Let $(N,h)$ be a Riemannian manifold and $y_0 \in N$, we define a bijection $\Lambda_{y_0}^{\nabla^{h}} (\cdot)$ from the set of absolutely continuous curves $\gamma : [0,1] \rarrow N$ starting at $y_0$ onto an open subset of the Banach space of absolutely continuous curves $[0, 1] \rarrow T_{y_0} N$ starting at 0, by
\begin{equation*}
\Lambda_{y_0}^{\nabla^{h}} (\gamma) (t) = \int_0^{t} (P^{\nabla^{h}})_{s}^{0} (\gamma) \dot{\gamma}(s) d s \text{    } \in T_{y_0} N, \text{    } \forall t \in [0,1].
\end{equation*}
It follows from Proposition \ref{p3} that the rolling curve with initial conditions $(\gamma, q_0)$ is given by:
$$
q_{\dr} (\gamma, q_0) (t) = (\gamma (t), {\hat{\Lambda}}_{\hx_0}^{-1} (A_0 \circ \Lambda_{x_0} (\gamma)) (t); P_0^{t} ({\hat{\Lambda}}_{\hx_0}^{-1} (A_0 \circ \Lambda_{x_0} (\gamma))) \circ A_0 \circ P_{t}^0 (\gamma)).
$$
Moreover, if the curve $\gamma$ is the geodesic on $M$ given by $\gamma (t) = exp_{x_0} (tX)$ with $\gamma (0) = x_0$ and $\dot{\gamma} (0) = X \in T_{x_0}M$, then, for $q_0 = (x_0 , \hx_0 ;A_0) \in Q$, the rolling curve $q_{\dr} (\gamma, q_0) : [0, a'] \rarrow Q$, $0 < a' \leq a$, is given by
$$
q_{\dr} (\gamma, q_0) (t) = (\gamma (t) , \hgamma_{\dr} (\gamma , q_0) (t) = \widehat{exp}_{\hx_0} (tA_0 X) , A_{\dr} (\gamma , q_0) (t) = P_0^{t} (\hgamma_{\dr} (\gamma , q_0)) \circ A_0 \circ P_{t}^0 (\gamma)).
$$
We also have that if $\hM$ is complete then $a=a'$.
\end{remark}

Let $\widehat{\lns}$ and $\widehat{\lr}$ (resp. $\widehat{\dns}$ and $\widehat{\dr}$) be the no-spinning and rolling lifts (resp. the no-spinning and rolling distributions), respectively, on $\hat{Q}:= Q (\hM,M)$. Thus, $\dim \widehat{\dns} = n + \hn = \dim \dns$ but, in contrary, $\dim \widehat{\dr} = \hn$, $\dim \dr=n$. This shows that the model of rolling of manifolds of different dimensions against each other is not symmetric with respect to $M$ and $\hM$.

\begin{proposition}\label{p12}
Let $\tbar$ the mapping defined by (\ref{DT}), we have the followings results:
\begin{enumerate}
\item[$1.$]
$\tbar_{*} \dns = \widehat{\dns}$,
\item[$2.$]
$\tbar_{*} V (\pq) = V (\pi_{\hat{Q}} )$,
\item[$3.$]
when $n \leq \hn$, we have $\tbar_{*} \dr \subset \widehat{\dr}$.
\end{enumerate}
\end{proposition}

\begin{proof}
We can assume, without loss of generality, that $n \leq \hn$.
\newline
\newline
$1.$\quad For $q_0 = (x_0, \hx_0; A_0) \in Q(M, \hM )$, let $\gamma$, $\hgamma$ be a smooth paths in $M$, $\hM$ starting at $x_0$, $\hx_0$, respectively, at t = 0. We have that $ (P_0^t (\hgamma) \circ A_0 \circ P_t^0 (\gamma))^{T} = P_0^t (\gamma) \circ A_0^{T} \circ P_t^0 (\hgamma)$, then
$$
\tbar (\gamma (t), \hgamma (t) ; P_0^t (\hgamma) \circ A_0 \circ P_t^0 (\gamma)) = (\hgamma (t) , \gamma (t);P_0^t (\gamma) \circ \tbar (x_0, \hx_0 ; A_0) \circ P_t^0 (\hgamma)).
$$
This immediately shows, by differentiating it with respect to $\frac{d}{dt} |_0$ and using the definition of $\lns$, that
$$
\tbar_{*} |_{q_0} \lns (X, \hX)|_{q_0} = \widehat{\lns} (\hX,X) |_{\tbar (q_0)},
$$
where $X = \dot{\gamma} (0)$, $\hX = \dot{\hgamma} (0)$. In particular, $\tbar_{*}$ maps $\dns$ isomorphically onto $\widehat{\dns}$.
\newline
\newline
$2.$\quad Let $\nu(B) |_{q = (x,\hx ; A)} \in V |_{q} (\pq)$, $B$ verifies $A^{T} B \in \mathfrak{so} (T_{x} M)$ then $\nu (B^{T}) |_{\tbar(q)} \in V |_{\tbar(q)} (\pi_{\hat{Q}})$. Then, $\tbar_{*} V (\pq) = V (\pi_{\hat{Q}} )$ because we have, for any $\hat{f} \in C^{\infty} (\hat{Q})$,
\begin{align*}
(\tbar_{*} \nu(B) |_{q}) \hat{f} = \nu (B) |_{q} (\hat{f} \circ \tbar) = \frac{d}{ds } |_{0} \hat{f}(\tbar (x,\hx; A + s B)) = \frac{d}{ds} |_0 \hat{f} (\hx,x; A^{T} + s B^{T}) = \nu (B^{T}) |_{\tbar (q)}.
\end{align*}
\newline
$3.$\quad For $q_0 = (x_0, \hx_0; A_0)$ and $X \in T_{x_0} M$, one has
$$
\tbar_{*} |_{q_0} \lr (X)|_{q_0} = \tbar_{*} |_{q_0} \lns (X, A_0 X)|_{q_0}
= \widehat{\lns} (A_0 X,A_0^T A_0X) |_{\tbar (q_0)} = \widehat{\lr} (A_0 X) |_{\tbar (q_0)},
$$
since $X = A_0^{T} (A_0 X) = \tbar (q_0) (A_0 X)$. Hence $\tbar$ maps $\dr$ of $Q(M, \hM )$ into $\widehat{\dr}$ of $Q( \hM ,M)$.

\end{proof}

\subsection{The Lie Brackets on $Q$}

Let $\Oh$ be an immersed submanifold of $\tsmthm$ and write $\pi_{\Oh} := \pi_{\tsmthm} |_{\Oh}$. If $\tbar : \Oh \rarrow T_{m}^{k} (M \times \hM) $ with $\pi_{T_{m}^{k} (M \times \hM)} \circ \tbar = \pi_{\Oh}$ (i.e. $\tbar \in C^{\infty} (\pi_{\Oh}, \pi_{T_{m}^{k} (M \times \hM)})$) and if $q =(x, \hx; A) \in \Oh$ and $\Xbar \in T_{(x,\hx)} (M \times \hM)$ such that $\lns (\Xbar) |_{q} \in T_{q} \Oh$, then we want to define what it means to take the derivative $\lns(\Xbar) |_{q} \tbar$. Our main interest will be the case where $k=1$, $m=0$ i.e. $T (M \times \hM)$, but some arguments below require a general setting. As a first step, we take $\Oh = \tsmthm$. We can inspire, from Eq. \eqref{D}, the following definition
\begin{equation*}
\lns(\Xbar) |_{q} \tbar := \nablabar_{\Xbar} (\tbar(\tilde{A})) - \nu(\nablabar_{\Xbar} \tilde{A})|_{q} \tbar \in {T_{m}^{k} (M \times \hM)}.
\end{equation*}

Here, $\tbar(A) = \tbar \circ A$ is a locally defined $(k,m)$-tensor field on $M \times \hM$. On the other hand, if $\overline{\omega} \in \Gamma (\pi_{T_{k}^{m} (M \times \hM)})$ and if we write $(\tbar \overline{\omega}) (q) := \tbar (q) \overline{\omega} |_{(x,\hx)}$ as a full contraction for $q = (x, \hx ; A) \in \tsmthm$, then we may compute
$$
\begin{array}{rl}
(\lns (\Xbar) |_{q} \tbar) \overline{\omega} = & (\nablabar_{\Xbar} (\tbar (A))) \overline{\omega} - (\frac{d}{dt} |_0 \tbar (A + t \nablabar_{\Xbar} A)) \overline{\omega}\\
\\
= & \nablabar_{\Xbar} (\tbar (A) \overline{\omega}) - \tbar (q) \nablabar_{\Xbar} \overline{\omega} - \frac{d}{dt} |_0 (\tbar (A + t \nablabar_{\Xbar} A) \overline{\omega} )\\
\\
= & \nablabar_{\Xbar} ((\tbar \overline{\omega}) (A)) - \frac{d}{dt} |_0 (\tbar \overline{\omega}) (A + t \nablabar_{\Xbar} A) - \tbar (q) \nablabar_{\Xbar} \overline{\omega}.
\end{array}
$$
Hence,
\begin{equation}\label{E}
(\lns (\Xbar) |_{q} \tbar) \overline{\omega} = \lns (\Xbar) |_{q} (\tbar \overline{\omega} )  - \tbar (q) \nablabar_{\Xbar} \overline{\omega}.
\end{equation}
Alternatively, Eq. \eqref{E} represents an intrinsic definition of $\lns (\Xbar) |_{q} \tbar$.

Now, if $\Oh \subset \tsmthm$ is an immersed submanifold, we could take Eq. \eqref{E} as the definition of $\lns (\Xbar) |_{q} \tbar$ for $q \in \Oh$.
\begin{definition}
Let $\Oh \subset \tsmthm$ be an immersed submanifold, $q= (x,\hx;A) \in \Oh$ and $\Xbar \in T_{(x,\hx)} (M \times \hM)$ be such that $\lns (\Xbar) |_{q} \in T_{q} \Oh$. Then for $\tbar : \Oh \rarrow T_{m}^{k} (M \times \hM)$ such that $\pi_{T_{m}^{k} (M \times \hM)} \circ \tbar = \pi_{\Oh}$, we define $\lns (\Xbar) |_{q} \tbar $ to be the unique element in $T_{m}^{k} |_{(x,\hx)} (M \times \hM)$ such that Eq. (\ref{E}) holds for every $\overline{\omega} \in \Gamma (\pi_{T_{k}^{m} (M \times \hM)})$ and call it the derivative of $\tbar$ with respect to $\lns(\Xbar) |_{q}$.
\end{definition}
We next present the main Lie brackets formulas obtained as in Proposition 3.45, Proposition 3.46, Proposition 3.47 of Section 3 in \cite{ChitourKokkonen}.
\begin{proposition}\label{p4}
Let $\Oh \subset \tsmthm$ be an immersed submanifold, $\tbar = (T , \hT)$, $\sbar = (S, \hS) \in C^{\infty} (\pi_{\Oh}, \pi_{T (M \times \hM)})$ be such that for all $\q \in \Oh$, $\lns (\tbar (q)) |_{q}$, $\lns (\sbar (q)) |_{q} \in T_{q} \Oh$ and $U$, $V \in C^{\infty} (\pi_{\Oh}, \pi_{\tsmthm})$,  be such that for all $\q \in \Oh$, $\nu (U(q)) \onq$, $\nu(V(q)) \onq \in T_{q} \Oh$. Then, one has
\newline
\newline
$1.$
\centerline{
\newline
$
\begin{array}{rc}
[\lns(\tbar (.)), \lns(\sbar(.))] \onq  = & \lns (\lns(\tbar (q)) \onq \sbar - \lns(\sbar (q)) \onq \tbar) \onq \\
 & +   \nu (A R(T(q), S(q)) - \hR (\hT(q) , \hS (q)) A) \onq,\\
\end{array}
$
}
\newline
\newline
\newline
\centerline{
$
[\lns(\tbar (.)), \nu(U(.))] \onq  = - \lns (\nu(U (q)) \onq \tbar) \onq + \nu (\lns( \tbar (q)) \onq U ) \onq,
$
}
\newline
\newline
$3.$
\centerline{
$
[\nu(U(.)) , \nu(V(.))] \onq = \nu (\nu(U(q)) \onq V - \nu(V(q)) \onq U)\onq.
$
}
\newline
\newline
Both sides of the equalities in $1.$ , $2.$ and $3. $ are tangent to $\Oh$.
\end{proposition}

\section{Rolling Orbits and Rolling Distributions}

In this section, we first characterize the rolling orbits corresponding to the ($NS$) and ($R$) problems and the we provide specific results on $\dr$-orbits in the case $\vert n-\hn\vert =1$.

\subsection{General Properties of Rolling Orbits}
We collect here some basic results on the structure of the orbits and the distributions of the two rolling systems.
To begin with, we completely describe the reachable sets of ($NS$) to the holonomy groups of the Riemannian manifolds $(M,g)$ and $(\hM, \hg)$, which are  Lie subgroups of $SO(n)$ and $SO(\hn)$.

In this setting,  $H |_{x}$ and $\hat{H} |_{\hx}$ denote $H^{\nabla} |_{x}$ and $H^{\hat{\nabla}} |_{\hx}$ respectively(for the notations, see the section \ref{sec:notations}).

The corresponding Lie algebras will be written as $\mathfrak{h} |_{x}$, $\hat{\mathfrak{h}} |_{\hx}$. 
Following the arguments of Theorem 4.1, Corollaries 4.2 and 4.3 and Proposition 4.5 of Section 4 as well as Property 5.2 of Section 5 in \cite{ChitourKokkonen}, one gets the subsequent result.
\begin{theorem}\label{t1}
Let $\qz \in Q$. Then the part of the orbit $\odns (q_0) $ of $\dns$ through $q_0$ that lies in the $\pq$-fiber over $(x_0, \hx_0)$ is
given by
\begin{eqnarray}\label{L}
\begin{array}{c}
\odns (q_0) \cap  \pq^{-1} (x_0, \hx_0) =  \{ \hh \circ A_0 \circ h \mid \hh \in \hH \onhxz, \; h \in H \onxz\} =: \hH \onhxz \circ A_0 \circ H \onxz,
\end{array}
\end{eqnarray}

In addition, at the tangent space level, we have
\begin{eqnarray}\label{DW}
\begin{array}{rl}
T_{q_0} \odns (q_0) \cap V \onqz (\pq) & = \nu (\{ \hat{k} \circ A_0 - A_0 \circ k \mid \;\;\; k \in \mathfrak{h} \onxz, \; \hat{k} \in \hat{\mathfrak{h}} \onhxz \}) \onqz\\
& =: \nu (\hat{\mathfrak{h}} \onhxz \circ A_0 - A_0 \circ  \mathfrak{h} \onxz) \onqz.
\end{array}
\end{eqnarray}
\end{theorem}

\begin{proposition}
If $\hM$ is complete, then
for every $q_0\in Q$, the map
$\pi_{\odr (q_0), M}:=\pi_{Q,M} |_{\odr (q_0)}:\odr (q_0) \rarrow M$
defines a smooth subbundle of $\pi_{Q,M}$.

\end{proposition}

We next compute the first commutators of $\lr(X)$ where $X \in \VF(M)$. The resulting formulas are obtained as in Proposition 5.9 of Section 5 in \cite{ChitourKokkonen}.

\begin{theorem}\label{t4}
If $X$, $Y \in \VF(M)$, $\q \in Q$, then
\begin{equation}\label{S}
[\lr(X), \lr(Y)] \onq = \lr ([X,Y]) \onq + \nu (A R (X,Y) - \hR (AX,AY)A)\onq.
\end{equation}
\end{theorem}

\begin{definition}
For $q=(x,\hat{x};A)\in Q$, we define the rolling curvature $\Rol_q$ at $q$ by
\[
\Rol_q(X,Y):= AR(X,Y ) - \hR (AX,AY) A,\quad X,Y\in T_xM.
\]
If $X,Y\in\VF(M)$, we write $\Rol(X,Y)$ for the map $Q\to \tsmthm$; $q\mapsto \Rol_q(X,Y)$.

Similarly, for $k\geq 0$, we define the $k$-th covariant derivative of $\Rol$ at $q$ by
\[
(\nablabar^{k} \Rol)_q(X,Y,Z_1,...,Z_{k}) := A(\nabla^{k} R)(X,Y,(.),Z_1,...,Z_{k}) - (\hnabla^{k} \hR)(AX,AY,A(.),AZ_1,...,AZ_{k}).
\]
\end{definition}

Clearly, for all $(x,\hx;A) \in Q$,
$$
A^{T} \Rol_q(X,Y), \; A^{T} (\nablabar^{k} \Rol)_q(X,Y,Z_1,...,Z_{k}) \in \mathfrak{so} (T_{x} M) \; if \; n \leq \hn,
$$
and
$$
\Rol_q(X,Y) A^{T}, \; (\nablabar^{k} \Rol)_q(X,Y,Z_1,...,Z_{k}) A^{T} \in \mathfrak{so} (T_{\hx} \hM)  \; if \; n \geq \hn,
$$
and therefore,
$\nu(\Rol_q(X,Y))$, $(\nablabar^{k} \Rol)_q(X,Y,Z_1,...,Z_{k})$ are well defined as elements of $V \onq (\pq)$.

\begin{remark}
With this notation, Eq. (\ref{S}) can be written as
\begin{equation}
[\lr(X), \lr(Y)] \onq = \lr ([X,Y]) \onq + \nu (\Rol_q(X,Y))\onq.
\end{equation}
\end{remark}

\begin{proposition}
Let $X$, $Y$, $Z \in \VF(M)$. Then, for $\q \in \tsmthm $, one has
\begin{align*}
[ \lr (Z), \nu (\Rol (X,Y))] \onq = & - \lns (\Rol (X,Y)Z) \onq + \nu \big( (\overline{\nabla}^1 \Rol)_q (X,Y,Z) \big) \onq \\
 & + \nu \big( \Rol_q(\nabla_Z X, Y)\big) \onq + \nu \big( \Rol_q(X, \nabla_Z Y) \big) \onq.
\end{align*}
\end{proposition}

We recall the following notation 
we define
$$
[A,B]_{\mathfrak{so}} := A \circ B - B \circ A \in \mathfrak{so}(T_{x} M).
$$

\begin{proposition}
Let $\q \in Q$ and $X$, $Y$, $Z$, $W \in \VF (M)$. We have
\begin{align*}
 & \big[\nu (\Rol(X,Y)), \nu (\Rol(Z,W))\big] \onq\\
= &  \nu \Big(A [R(X,Y),R(Z,W)]_{\mathfrak{so}} - [\hR ( AX,AY), \hR( AZ, A W)]_{\mathfrak{so}} A -  \hR ( \Rol_q (X,Y) Z,AW) (A)\\
&  - \hR ( AZ, \Rol_q (X,Y) W) (A)  +  \hR ( A X, \Rol_q (Z,W) Y) (A) +  \hR (\Rol_q (Z,W) X, AY)\Big)\onq.
\end{align*}
\end{proposition}

\begin{proof}
Cf. the proof of Proposition 5.18 and Corollary 5.19 of Section 5 in \cite{ChitourKokkonen}.
\end{proof}

\begin{proposition}
Consider the following smooth right and left actions of $Iso (M,g)$ and $Iso (\hM, \hg)$ on $Q$ given by
$$
q_0 \cdot F := (F^{-1} (x_0), \hx_0; A_0 \circ F_{*}|_{F^{-1} (x_0)}), \quad  \hF \cdot q_0:= (x_0, \hF (\hx_0); \hF_{*}|_{\hx_0} \circ A_0),
$$
where $\qz \in Q$, $F \in Iso (M,g)$ and $\hF \in Iso (\hM, \hg)$. We also set
\[
\hF \cdot q_0 \cdot F :=(\hF \cdot q_0) \cdot F = \hF \cdot (q_0 \cdot F).
\]
Then for any $\qz \in Q$, absolutely continuous $\gamma: [0,1] \rarrow M$ such that $\gamma (0) = x_0$, $F \in Iso (M,g)$ and $\hF \in Iso (\hM, \hg)$, we have
\begin{align}\label{DN}
\hF \cdot q_{\dr} (\gamma, q_0)(t) \cdot F = q_{\dr} (F^{-1} \circ \gamma, \hF \cdot q_0 \cdot F) (t),
\end{align}
for all $t \in [0,1]$. In particular, $\hF \cdot \odr ( q_0) \cdot F = \odr (\hF \cdot q_0 \cdot F)$.
\end{proposition}

\begin{proof}
Cf. the proof of Proposition 5.5 of Section 5 in \cite{ChitourKokkonen}.
\end{proof}

\begin{remark}
When $n \leq \hn$, the right action of $Iso (M,g)$ on $Q$ is free. Indeed, given $F$, $F' \in Iso (M,g)$, the existence of an $\q \in Q$ such that $ q \cdot F = q \cdot F' $ implies that $F^{-1} (x) = F'^{-1} (x): = y $ and $A \circ F_{*} |_{y} = A \circ {F'}_{*} |_{y}$. Since $A^{T} A = \id$, we obtain $F_{*} |_{y} = {F'}_{*} |_{y}$,
which implies, because $M$ is connected,
that $F = F'$ (see \cite{Sakai}, page 43). The same argument proves the freeness of the left $Iso (\hM, \hg)$-action when $n \geq \hn$.
\end{remark}

\subsection{Elementary Constructions when $\mid n - \hn \mid = 1$}

\begin{proposition}\label{p15}
Let $(M,g)$ and $(\hM,\hg)$ be Riemannian manifolds of dimensions $n$ and $\hn = n-1$ respectively, with $n \geq 2$. We use $(\hM^{(1)},\hg^{(1)})$ to denote the Riemannian product
$(\mathbb{R} \times \hM, dr^2 \oplus \hg)$,
where $dr^2$ denotes the canonical Riemannian metric on $\mathbb{R}$.

Set $Q^{(1)} := Q(M,\hM^{(1)})$ and let $\lr^{(1)}$, $\dr^{(1)}$ to be the rolling lift and the rolling distribution on $Q^{(1)}$. We define, for every $a \in \mathbb{R}$,
$$
\iota_{a} : Q \rarrow Q^{(1)}; \;\;\; \iota_{a} (x,\hx;A)= (x, (a,\hx); A^{(1)}),
$$
where $A^{(1)}: T_{x} M \rarrow T_{(a,\hx)} (\mathbb{R} \times \hM)$ is defined as follows: $A^{(1)} \in Q^{(1)}$,
$$
A^{(1)} |_{(\ker A)^{\perp}} = (0, A |_{(\ker A)^{\perp}}), \quad A^{(1)} (\ker A) = \mathbb{R}
\partial_r |_{(a,\hx)} \times \{0\},
$$
where $\partial_r$
is the canonical vector field on $\mathbb{R}$ in the positive direction, also seen as a vector field on $\hM^{(1)}$ in the usual way.\\
Then for every $a \in \mathbb{R}$, the map $\iota_{a}$ is an embedding and for every $\qz \in Q$, $a_0 \in \mathbb{R}$ and $X \in T_{x} M$, one has
$$
\lr (X) \onqz = \Pi_{*} \lr^{(1)} (X) |_{\iota_{a_0} (q_0)},
$$
$$
\odr(q_0) = \Pi (\Oh_{\dr^{(1)}} (\iota_{a_0} (q_0))),
$$
where
$$
\begin{array}{lccc}
\Pi: & Q^{(1)} & \rarrow & Q;\\
& (x,(a,\hx);A^{(1)}) & \mapsto & (x,\hx; (pr_2)_{*} \circ A^{(1)}),
\end{array}
$$
is a surjective submersion and $pr_2 : \mathbb{R} \times \hM \rarrow \hM$ is the projection onto the second factor.
\end{proposition}

\begin{proof}
Let $\gamma$ be a path in $M$ starting at $x_0$ and $q(t) = (\gamma (t), \hgamma (t); A(t)) : = q_{\dr} (\gamma, q_0) (t)$. We define a path $q^{(1)} (t) = (\gamma (t), \hgamma^{(1)} (t) ; A^{(1)} (t)) $ on $Q^{(1)}$ as follows:
$$
\hgamma^{(1)} (t) := \big(a_0 + \int_0^{t} \iota_{a_0} (A_0) p^{T} (A_0) P_{s}^0 (\gamma) \dot{\gamma} (s) ds, \hgamma (t)\big), \quad A^{(1)} := P_0^{t} (\hgamma^{(1)}) \circ \iota_{a_0} (A_0) \circ P_{t}^0 (\gamma),
$$
where, for every $\q \in Q$, we define the $g$-orthogonal projections as
$$
p^{\perp} (A) : T_{x} M \rarrow (\ker A)^{\perp}, \quad p^{T} (A) : T_{x} M \rarrow \ker A.
$$
We will show that $q^{(1)}$ is the rolling curve on $Q^{(1)}$ starting from $\iota_{a_0} (q_0)$. Indeed, clearly $q^{(1)} (0) = (\gamma (0), (a_0,\hgamma (0)); \iota_{a_0} (A_0)) = \iota_{a_0} (q_0)$ and $A^{(1)} (t) \in Q^{(1)}$ for every time $t$ and $\iota_{a_0} (A_0) \in Q^{(1)}$. We also have
$$
\dot{\hgamma}^{(1)} (t) = (b (t) \partial_r |_{\hgamma^{(1)} (t)} , \dot{\hgamma} (t)),
$$
where $ b(t)$ is defined by $\iota_{a_0} (A_0) p^{T} (A_0) P_{t}^0 (\gamma) \dot{\gamma} (t) := (b (t) \partial_r |_{(a_0,\hx_0)} , 0)$. On the other hand,
$$
\begin{array}{rl}
A^{(1)} (t) \dot{\gamma} (t) & = P_0^{t} (\hgamma^{(1)}) \iota_{a_0} (A_0) P_{t}^0 (\gamma) \dot{\gamma} (t)\\
& = P_0^{t} (\hgamma^{(1)}) \iota_{a_0} (A_0) (p^{T} (A_0) + p^{\perp} (A_0)) P_{t}^0 (\gamma) \dot{\gamma} (t).
\end{array}
$$
Since $\hM^{(1)}$ is a Riemannian product, then, for every $\hX  \in T_{\hx_0} \hM\subset T_{(a_0,\hx_0)} (\mathbb{R} \times \hM)$, we have
\begin{align*}
P_0^{t} (\hgamma^{(1)}) (0, \hX ) = (0, P_0^{t} (\hgamma) \hX), \quad P_0^{t} (\hgamma^{(1)}) ( \partial_r |_{(a_0,\hx_0)} , 0 )= (\partial_r |_{\hgamma^{(1)} (t)} , 0).
\end{align*}
However $\iota_{a_0} (A_0)  p^{\perp} (A_0) X = (0, A_0 p^{\perp} (A_0) X ) = (0, A_0 X)$ for every $X \in T_{x_0} M$, we get that
$$
P_0^{t} (\hgamma^{(1)}) \iota_{a_0} (A_0) p^{\perp} (A_0) P_{t}^0 (\gamma) \dot{\gamma} (t) = P_0^{t} (\hgamma^{(1)}) ( 0 , A_0 P_{t}^0 (\gamma) \dot{\gamma} (t) ) = (0, P_0^{t} (\hgamma) A_0 P_{t}^0 (\gamma) \dot{\gamma} (t)),
$$
and $P_0^{t} (\hgamma^{(1)}) \iota_{a_0} (A_0) p^{T} (A_0)  P_{t}^0 (\gamma) \dot{\gamma} (t) = P_0^{t} (\hgamma^{(1)}) (b (t) \partial_r |_{(a_0,\hx_0)} , 0) = (b (t) \partial_r |_{\hgamma^{(1)} (t)} , 0)$.
Therefore,
$$
\begin{array}{rl}
A^{(1)} (t) \dot{\gamma} (t) & = P_0^{t} (\hgamma^{(1)}) \iota_{a_0} (A_0) (p^{T} (A_0) + p^{\perp} (A_0)) P_{t}^0 (\gamma) \dot{\gamma} (t)\\
& =  ( b(t) \partial_r |_{\hgamma^{(1)} (t)},0) + (0, P_0^{t} (\hgamma) A_0 P_{t}^0 (\gamma) \dot{\gamma} (t))\\
& =  ( b(t) \partial_r |_{\hgamma^{(1)} (t)}, A (t) \dot{\gamma} (t))\\
& =  ( b(t) \partial_r |_{\hgamma^{(1)} (t)}, \dot{\hgamma} (t)) = \dot{\hgamma}^{(1)} (t).\\
\end{array}
$$
This and the definition of $A^{(1)} (t)$ show that $q^{(1)} (t) = q_{\dr^{(1)}} (\gamma, \iota_{a_0} (q_0)) (t)$ for all $t$.\\
Furthermore, since $A^{(1)} (t) \dot{\gamma} (t) =  P_0^{t} (\hgamma^{(1)}) \iota_{a_0} (A_0) (p^{T} (A_0) + p^{\perp} (A_0)) P_{t}^0 (\gamma) \dot{\gamma} (t)$ and by the basic properties of parallel transport, it follows that
$$
\begin{array}{rl}
\Pi (q_{\dr^{(1)}} (\gamma, \iota_{a_0} (q_0)) (t))  = \Pi(\gamma (t), \hgamma^{(1)}; A^{(1)}) & = (\gamma (t), \hgamma (t); (pr_2)_{*} \circ A^{(1)})\\
& = (\gamma (t), \hgamma (t); (pr_2)_{*} ( P_0^{t} (\hgamma^{(1)}) \circ \iota_{a_0} (A_0) \circ P_{t}^0 (\gamma)))\\
& = (\gamma (t), \hgamma (t); P_0^{t} (\hgamma) A_0 P_{t}^0 (\gamma)) = q_{\dr} (\gamma, q_0) (t).
\end{array}
$$
Hence $\odr (q_0) \subset \Pi (\Oh_{\dr^{(1)}} (\iota_{a_0} (q_0)))$ as well as
$$
\Pi_{*} (\lr^{(1)} (\dot{\gamma} (0) )|_{\iota_{a_0} (q_0)}) = \Pi_{*} (\dot{q}_{\dr^{(1)}} (\gamma, \iota_{a_0} (q_0)) (0) ) = \dot{q}_{\dr} (\gamma, q_0) (0) = \lr(\dot{\gamma} (0)) |_{q_0}.
$$
Finally, if $q^{(1)} = (x,(a,\hx);A^{(1)}) \in \Oh_{\dr^{(1)}} (\iota_{a_0} (q_0))$, take a path $\gamma$ in $M$ starting from $x_0$ such that $q^{(1)} = q_{\dr^{(1)}} (\gamma, \iota_{a_0} (q_0)) (1)$. By what was done above, it follows that $\Pi (q_{\dr^{(1)}} (\gamma, \iota_{a_0} (q_0)) (t))= q_{\dr} (\gamma, q_0) (t)$ and thus, evaluating this at $t=1$ gives $\Pi (q^{(1)}) \in \odr (q_0)$, whence $\Pi (\Oh_{\dr^{(1)}} (\iota_{a_0} (q_0)))$ $\subset$ $\odr (q_0)$.
The claim that $\iota_{a}$ is an embedding for every $a \in \mathbb{R}$ and $\Pi $ is a surjective submersion are obvious from the fact $\Pi \circ \iota_{a} = \id_{Q}$.
\end{proof}

\begin{corollary}\label{c2}
With the same notations of the previous proposition, if the orbit $\odr (q_0)$ is not open in $Q$ for some $q_0 \in Q$, then $\Oh_{\dr^{(1)}} (\iota_{a_0} (q_0))$ is not open in $Q^{(1)}$.
\end{corollary}

\begin{proof}
Suppose $\Oh_{\dr^{(1)}} (\iota_{a_0} (q_0))$ were open in $Q^{(1)}$, then since $\Pi : Q^{(1)} \rarrow Q$ is a smooth submersion, it is an open map and hence its image $\Pi (\Oh_{\dr^{(1)}} (\iota_{a_0} (q_0))) = \odr (q_0)$ is open.
\end{proof}

With the assumption and the notations of Proposition \ref{p15}, we have the following remark.
\begin{remark}\label{r3}
Keeping the same notations as before, recall that $Q = Q(M , \hM)$ is connected and thus, as a consequence of Corollary \ref{c2}, if the system associated to the rolling of $M$ and $\hM^{(1)}$ is controllable then the system associated to the rolling of $M$ and $\hM$ is also controllable.
\end{remark}

\begin{proposition}\label{p16}
Let $(M,g)$ and $(\hM,\hg)$ be Riemannian manifolds of dimensions $n = \hn -1$ and $\hn$, with $\hn \geq 2$ respectively. Let $(M^{(1)}, g^{(1)})$ be the Riemannian product $(\mathbb{R} \times M, dr^2 \oplus g)$, with the obvious orientation. Write $Q^{(1)} = Q(M^{(1)}, \hM)$ and let $\lr^{(1)}$, $\dr^{(1)}$ be the rolling lift and the rolling distribution on $Q^{(1)}$. We define for every $a \in \mathbb{R}$,
$$
\iota_{a} : Q \rarrow Q^{(1)}; \;\;\; \iota_{a} (x, \hx;A) = ((a,x), \hx;A^{(1)}),
$$
where $A^{(1)}: T_{(a,x)} (\mathbb{R} \times M) \rarrow T_{\hx} \hM$ is defined as follows: $A^{(1)} \in Q^{(1)}$,
$$
A^{(1)} |_{T_{x} M} = A, \quad A^{(1)} \partial_r |_{(a,x)} \in (\IM A)^{\perp}.
$$

Then for every $a \in \mathbb{R}$, the map $\iota_{a}$ is an embedding and for every $\qz \in Q$, $a_0 \in \mathbb{R}$ and $X \in T_{x} M \subset T_{(a,x)} (\mathbb{R} \times M)$, one has
$$
(\iota_{a_0})_{*} \lr (X) \onqz = \lr^{(1)} (X) |_{\iota_{a_0} (q_0)}.
$$
Moreover, if one defines
$$
\begin{array}{cccl}
\Pi: & Q^{(1)} & \rarrow & Q;\\
& ((a,x),\hx;A^{(1)}) & \mapsto &  (x,\hx; A^{(1)} \circ (i_{a})_{*} ),
\end{array}
$$
where $i_{a} : M \rarrow \mathbb{R} \times M$; $x \mapsto (a,x)$ and if $\Delta_{R}$ is the subdistribution of $\dr^{(1)}$ defined by
$$
\Delta_{R} |_{q^{(1)}} = (\iota_{a})_{*} \dr |_{\Pi (q^{(1)})}, \;\;\; \forall q^{(1)}= ((a,x), \hx;A^{(1)}) \in Q^{(1)},
$$
then $\iota_{a_0} (\odr(q_0)) = \Oh_{\Delta_{R}} (\iota_{a_0} (q_0)) \subset \Oh_{\dr^{(1)}} (\iota_{a_0} (q_0))$.

\end{proposition}

\begin{proof}
The facts that $\iota_{a}$ is an embedding and $\Pi$ is submersion simply follow from the fact $\Pi \circ \iota_{a} = \id_{Q}$. Let now $\gamma$ be a path in $M$ starting from $x_0$ and $ q (t) = (\gamma (t) , \hgamma (t) ; A (t)) = q_{\dr} (\gamma, q_0) (t)$. We define a path $q^{(1)} (t) = (\gamma^{(1)} (t), \hgamma (t); A^{(1)} (t))$ on $Q^{(1)}$ by
$$
\gamma^{(1)} (t) := (a_0, \gamma (t)), \quad A^{(1)} := P_0^{t} (\hgamma) \circ \iota_{a_0} (A_0) \circ P_{t}^0 (\gamma^{(1)}),
$$
We will show that $q^{(1)}$ is the rolling curve on $Q^{(1)}$ starting from $\iota_{a_0} (q_0)$. Indeed, clearly $q^{(1)} (0) = ((a_0,\gamma (0)), \hgamma (0); \iota_{a_0} (A_0)) = \iota_{a_0} (q_0)$ and for $\iota_{a_0} (A_0) \in Q^{(1)}$ we have $A^{(1)} (t) \in Q^{(1)}$ for all $t$. We also have $
\dot{\gamma}^{(1)} (t) = (0 , \dot{\gamma} (t))$. On the other hand,
$$
\begin{array}{rl}
A^{(1)} (t) \dot{\gamma}^{(1)} (t) & = P_0^{t} (\hgamma) \iota_{a_0} (A_0) P_{t}^0 (\gamma^{(1)}) \dot{\gamma}^{(1)} (t)\\
& = P_0^{t} (\hgamma) \iota_{a_0} (A_0)  P_{t}^0 (\gamma^{(1)}) (0 , \dot{\gamma} (t)).
\end{array}
$$
Since $M^{(1)}$ is a Riemannian product, then $P_{t}^0 (\gamma^{(1)}) (0, X ) = (0, P_{t}^0 (\gamma) X)$ for every $X  \in T_{x_0} M \subset T_{(a_0,x_0)} (\mathbb{R} \times M)$. Therefore,
$$
\begin{array}{rl}
A^{(1)} (t) \dot{\gamma}^{(1)} (t) & = P_0^{t} (\hgamma) \iota_{a_0} (A_0) (0, P_{t}^0 (\gamma)  \dot{\gamma} (t))\\
& =  P_0^{t} (\hgamma) A_0 P_{t}^0 (\gamma)  \dot{\gamma} (t)\\
& =  A(t) \dot{\gamma} (t) = \dot{\hgamma} (t).
\end{array}
$$
This proves that $q^{(1)} (t) = q_{\dr^{(1)}} (\gamma^{(1)}, \iota_{a_0} (q_0)) (t)$ for all $t$.
Furthermore, notice that $\pi_{Q^{(1)}} (\iota_{a_0} (q(t))) = ((a_0, \gamma (t)), \hgamma (t)) = (\gamma^{(1)} (t), \hgamma (t)) = \pi_{Q^{(1)}} (q^{(1)} (t))$ and $A^{(1)} (t) (0,X) = A(t) X = \iota_{a_0} (A(t)) X $ for every $X \in T_{x} M \subset T_{(a_0,x)} (\mathbb{R} \times M)$. However
$
A^{(1)} (t) T_{\gamma (t)} M \perp A^{(1)} (t) \partial_r |_{\gamma^{(1)} (t)}
$ and $(\iota_{a_0} \circ A (t) ) T_{x} M \perp (\iota_{a_0} \circ A (t) ) \partial_r |_{\gamma^{(1)} (t)}$, we must have, by orientation,
$
A^{(1)} (t)\partial_r |_{\gamma^{(1)} (t)} = (\iota_{a_0} \circ A (t) ) \partial_r |_{\gamma^{(1)} (t)}.
$
This proves that $\iota_{a_0} (q(t)) = q^{(1)} (t)$ and hence
$$
(\iota_{a_0})_{*} \lr (\dot{\gamma} (0)) |_{q_0} = (\iota_{a_0})_{*} \dot{q}(0) = \dot{q}^{(1)} (0) =  \lr^{(1)} (\dot{\gamma}^{(1)} (0)) |_{\iota_{a_0} (q_0)} = \lr^{(1)} ((0, \dot{\gamma} (0))) |_{\iota_{a_0} (q_0)}.
$$
So, $(\iota_{a_0})_{*} \lr (X) |_{q_0} = \lr^{(1)} (X) |_{\iota_{a_0} (q_0)}$ for every $X \in T_{x_0} M \subset T_{(a_0,x_0)} (\mathbb{R} \times M)$, then,
$$
\iota_{a_0} (\odr(q_0)) \subset \Oh_{\dr^{(1)}} (\iota_{a_0} (q_0)).
$$
Finally, recall that $\Pi \circ \iota_{a} = \id_{Q}$, then, for every $q \in \odr (q_0)$, we have
$$
\Delta_{R} |_{\iota_{a_0} (q)} = (\iota_{a_0})_{*} \dr \onq \subset T_{\iota_{a_0} (q_0)} (\iota_{a_0} (\odr (q_0))).
$$
Thus, one can write $\Delta_{R} |_{\iota_{a_0} (\odr (q_0))} = (\iota_{a_0})_{*} \dr |_{\odr (q_0)}$. Then,
$
\iota_{a_0} (\odr (q_0)) \subseteq \Oh_{\Delta_{R}} (\iota_{a_0} (q_0)).
$
Since $\iota_{a_0} |_{\odr (q_0)}$ is an immersion, we get the equality
$
\iota_{a_0} (\odr (q_0)) = \Oh_{\Delta_{R}} (\iota_{a_0} (q_0)).
$

\end{proof}

\begin{corollary}\label{c3}
With the assumptions of the previous proposition, if the orbit $\odr (q_0)$ is open in $Q$ for some $q_0 \in Q$, then the codimension of $\Oh_{\dr^{(1)}} (\iota_{a_0} (q_0))$ in $Q^{(1)}$ is at most 1.
\end{corollary}

\begin{proof}
The relation between the dimension of $Q$ and that of $Q^{(1)}$ is
$$
\dim Q = 2 \hn -1 + \frac{\hn (\hn -1)}{2} = \dim Q^{(1)} - 1.
$$
On the other hand, if $\odr(q_0)$ is open in $Q$ then one has $\dim \odr (q_0) = \dim Q$. Thus,
$$
\dim \Oh_{\dr^{(1)}} (\iota_{a_0} (q_0)) \geq \dim \Oh_{\Delta_{R}} (\iota_{a_0} (q_0)) = \dim \iota_{a_0} (\odr (q_0)) = \dim Q = \dim Q^{(1)} - 1.
$$
\end{proof}

\begin{theorem}\label{t9}
Let $M$ and $\hM$ be Riemannian manifolds of dimension $n=3$ and $\hn = 2$ respectively. If, for some $\qz \in Q$, the orbit $\odr (q_0)$ is not open in $Q$, then there exists an open dense subset $O$ of $\odr(q_0)$ such that for every $\qo \in O$ there is an open neighbourhood $U$  of $x_1$ for which it holds that $(U, g \mid_{U})$ is isometric to some warped product $(I \times N, h_{f})$, where $I \subset \mathbb{R}$ is an open interval and the warping function $f$ satisfying $f'' =0$.
\end{theorem}

\begin{proof}
We will proceed by using Proposition \ref{p15}. Let $(M^{(1)},g^{(1)})$ be the Riemannian product $(\mathbb{R} \times \hM , dr^2 \oplus \hg)$ and let $a_0 \in \mathbb{R}$. Since the orbit $\odr (q_0)$ is not open in $Q$, it follows from Corollary \ref{c2} that $\Oh_{\dr^{(1)}} (\iota_{a_0} (q_0))$ is not open in $Q^{(1)}$. Theorem 7.1 of Section 7 in \cite{ChitourKokkonen} provides an open subset $O^{(1)}$ of $\Oh_{\dr^{(1)}} (\iota_{a_0} (q_0))$ such that one of $(a) - (c)$ of this theorem holds. So, $O := \Pi (O^{(1)})$ is a dense open of $\odr (q_0)$ and let$\qo \in O$, then choose $q_1^{(1)} \in O^{(1)}$ such that $\Pi (q_1^{(1)}) = q_1$, whence $q_1^{(1)} = \iota_{a_1} (q_1)$ for some $a_1 \in \mathbb{R}$. Moreover, if $U$ and $\hU^{(1)}$ are the neighborhoods of $x_1$ and $(a_1,\hx_1)$, respectively, as in Theorem 7.1 mentioned before, then we can choose $\hU^{(1)}$ to be of the form $I \times \hU$ for some open interval $I \subset \mathbb{R}$ and open neighborhood $\hU \subset \hM$ of $\hx_1$. We consider the possible subcases.
\\
If $(a)$ holds, then $(U, g\mid_{U})$ is (locally) isometric to the Riemannian product $I \times \hU$, hence we have $f=1$.
If $(b)$ holds, then $(U, g\mid_{U})$ and $(\hU^{(1)}, g^{(1)} \mid_{\hU^{(1)}})$ are both of class $\mathcal{M}_{\beta}$ for some $\beta > 0$, but $(\hU^{(1)}, g^{(1)} \mid_{\hU^{(1)}})$ is as a Riemannian product, so it cannot be of such class $\mathcal{M}_{\beta}$, thus this case cannot occur.
If $(c)$ holds, let $F: (I \times N, h_{f}) \rarrow U$ and $\hF: (\hat{I} \times \hat{N}, \hh_{\hf}) \rarrow \hU$ be the isomorphisms, it means that $(\hat{I} \times \hat{N}, \hh_{\hf})$ is isomorphic to a Riemannian product which implies $\hf$ must satisfy $\hf'' =0$ thus also $f'' =0$.
\end{proof}

\section{Controllability Results}

\subsection{The Rolling Problem $\Sigma_{NS}$}

We start by the following remark about the non-compatibility of the ($NS$) system in the space $\tsmthm$.
\begin{remark}
The result of Theorem (\ref{t1}) can obviously be formulated in the space $\tsmthm$ instead of $Q$. Thus, it implies that each orbit $\odns (q_0)$ of $\dns $ in $\tsmthm$, $\qz \in \tsmthm$, has dimension at most $n + \hn + \dim H \onxz + \dim \hH \onhxz \leq n + \hn + \frac{n (n-1)}{2} + \frac{\hn (\hn-1)}{2} $. Let $N = max\{ n , \hn \}$ and $r = \mid \hn - n \mid$. Since the dimension of $\tsmthm$ is $n + \hn + n\hn$, then
$$
\hbox{codim} \odns (q_0) \geq n \hn - \frac{n (n-1)}{2} - \frac{\hn (\hn-1)}{2} = \frac{N-r^2} {2} > \lceil \frac{N-r^2} {2} \rceil,
$$
where $\lceil k  \rceil$ stands for the integer part of a real number $k$,
which means that $ \hbox{codim} \odns (q_0) \geq \lceil \frac{N-r^2} {2} \rceil + 1$, i.e., $\dns$ in never completely controllable in $\tsmthm$.
\end{remark}

Theorem \ref{t1} states that the controllability of $\dns$ is completely determined by the holonomy groups of $M$ and $\hM$. The next theorem highlights that fact at the Lie algebraic level.

\begin{theorem}\label{t2}
Fix some orthonormal frames $F$, $\hF$ of $M$, $\hM$ at $x$ and $\hx$ respectively. Let $\mathfrak{h} := \mathfrak{h} |_{F} \subset \mathfrak{so} (n)$ and $\hat{\mathfrak{h}} := \hat{\mathfrak{h}} |_{\hF} \subset \mathfrak{so} (\hn)$ be the holonomy Lie algebras of $M$ and $\hM$ with respect to these frames. Then the control system $(\sum)_{NS}$ is completely controllable if and only if for every $A \in SO (n,\hn)$ (defined in (\ref{B})),
\begin{eqnarray}
\hat{\mathfrak{h}} A - A \mathfrak{h}=
\left\lbrace
\begin{array}{lr}
\{ B \in (\mathbb{R}^{n})^{*} \otimes \mathbb{R}^{\hn} \mid A^{T} B \in \mathfrak{so} (n) \}, & if \; n < \hn, \\
\{ B \in (\mathbb{R}^{n})^{*} \otimes \mathbb{R}^{\hn} \mid B A^{T} \in \mathfrak{so} (\hn) \}, & if \; n > \hn.
\end{array}
\right.
\end{eqnarray}
\end{theorem}

\begin{proof}
By connectedness of $Q$, we get that $\dns $ is controllable if and only if every $\odns (q)$, $\q \in Q$, is open in $Q$. Clearly, an orbit $\odns (q_0) = Q$, $\qz \in Q$, is an open subset of $Q$ if and only if $T_{q} \odns (q_0) = T_{q} Q$ for some (and hence every) $q \in \odns (q_0)$. Thus the decomposition given by Remark \ref{r1} implies that an orbit $\odns (q_0)$ is open in $Q$ if and only if $V |_{q} (\pq) \subset T_{q} \odns (q_0)$ for some $q \in \odns (q_0)$.\\
Fix $(x_0, \hx_0) \in M \times \hM$. Theorem \ref{t1} implies that every $\dns$-orbit intersects every $\pq$-fiber. Hence $\dns$ is controllable if and only if $V |_{q} (\pq) \subset T_{q} \odns (q_0)$ for every $q = (x_0, \hx_0; A) \in Q |_{(x_0, \hx_0)}$. By (\ref{DW}), this condition is equivalent to the condition that, for every $q = (x_0, \hx_0; A) \in Q |_{(x_0, \hx_0)}$,
$$
\nu (\hat{\mathfrak{h}} \onhxz \circ A - A \circ  \mathfrak{h} \onxz) \onqz = V |_{q} (\pq).
$$
By Proposition \ref{p9}, one can deduces that, for every $q \in Q$,
$$
V |_{q} (\pq) =
\left\lbrace
\begin{array}{lr}
\nu(\{B \in T^*_{x_0} M \otimes T_{\hx_0}\hat{M} \mid A^{\overline{T}} B \in \mathfrak{so} (T_{x_0} M) \}) \onq, & if \; n \leq \hn,\\
\nu(\{B \in T^*_{x_0} M \otimes T_{\hx_0}\hat{M} \mid BA^{\overline{T}} \in \mathfrak{so} (T_{\hx_0} \hM) \}) \onq, & if \; n \geq \hn.
\end{array}
\right.
$$
Thus, we conclude that $\dns$ is controllable if and only if, for all $q = (x_0, \hx_0; A) \in Q |_{(x_0, \hx_0)}$
$$
\hat{\mathfrak{h}} \onhxz \circ A - A \circ  \mathfrak{h} \onxz =
\left\lbrace
\begin{array}{lr}
\{B \in T^*_{x_0} M \otimes T_{\hx_0}\hat{M} \mid A^{\overline{T}} B \in \mathfrak{so} (T_{x_0} M) \}, & if \; n \leq \hn,\\
\{B \in T^*_{x_0} M \otimes T_{\hx_0}\hat{M} \mid BA^{\overline{T}} \in \mathfrak{so} (T_{\hx_0} \hM) \}, & if \; n \geq \hn.
\end{array}
\right.
$$
Choosing arbitrary orthonormal local frames $F$ and $\hF$ of $M$ and $\hM$ at $x_0$ and $\hx_0$, respectively, we see that the above condition is equivalent to
$$
\hat{\mathfrak{h}} |_{\hF} \mathcal{M}_{F, \hF} (A) - \mathcal{M}_{F, \hF} (A) \mathfrak{h} |_{F}=
\left\lbrace
\begin{array}{lr}
\{B \in (\mathbb{R}^{n})^{*} \otimes \mathbb{R}^{\hn} \mid \mathcal{M}_{F, \hF} (A)^{\overline{T}} B \in \mathfrak{so} (n) \}, & if \; n \leq \hn,\\
\{B \in (\mathbb{R}^{n})^{*} \otimes \mathbb{R}^{\hn} \mid B\mathcal{M}_{F, \hF} (A)^{\overline{T}} \in \mathfrak{so} (\hn) \}, & if \; n \geq \hn.
\end{array}
\right.
$$
Since we have $\{\mathcal{M}_{F, \hF} (A) \mid A \in Q |_{(x_0,\hx_0)} \} = SO (n, \hn)$, $\tsmthm \cong (\mathbb{R}^{n})^{*} \otimes \mathbb{R}^{\hn}$ and $F, \hF$ were arbitrary chosen, the claim follows.
\end{proof}

\begin{theorem}\label{t3}
Suppose that $M$, $\hM$ are simply connected. Then $(\Sigma)_{NS}$ is completely controllable if and only if
\begin{eqnarray}\label{Q}
\hat{\mathfrak{h}} I_{n,\hn} - I_{n,\hn} \mathfrak{h}=
\left\lbrace
\begin{array}{lr}
\{ B \in (\mathbb{R}^{n})^{*} \otimes \mathbb{R}^{\hn} \mid I_{n,\hn}^{T} B \in \mathfrak{so} (n) \}, & if \; n \leq \hn, \\
\{ B \in (\mathbb{R}^{n})^{*} \otimes \mathbb{R}^{\hn} \mid B I_{n,\hn}^{T} \in \mathfrak{so} (\hn) \}, & if \; n \geq \hn.
\end{array}
\right.
\end{eqnarray}
\end{theorem}

\begin{proof}
Notice that $I_{n,\hn} \in SO (n,\hn)$, then the previous theorem give the necessary condition.\\
Conversely, suppose that the condition (\ref{Q}) holds. This condition implies that for $(x_0,\hx_0) \in M \times \hM$, there is an $\qz \in Q |_{(x_0,\hx_0)}$ such that
$$
\hat{\mathfrak{h}} A_0 - A_0 \mathfrak{h}=
\left\lbrace
\begin{array}{lr}
\{ B \in (\mathbb{R}^{n})^{*} \otimes \mathbb{R}^{\hn} \mid A_0^{T} B \in \mathfrak{so} (n) \}, & if \; n \leq \hn, \\
\{ B \in (\mathbb{R}^{n})^{*} \otimes \mathbb{R}^{\hn} \mid B A_0^{T} \in \mathfrak{so} (\hn) \}, & if \; n \geq \hn.
\end{array}
\right.
$$
By Proposition \ref{p9} and the equality \ref{DW}, this means that $T_{q_0} \odns (q_0) \cap V \onqz (\pq) = V \onqz (\pq) $ and hence $T_{q_0} \odns (q_0) = T_{q_0} Q$ due to Remark \ref{r1}. Thus $\odns (q_0)$ is open in $Q$. By the connectedness of $Q$, we have that $\odns (q_0) = Q$. Therefore, $(\Sigma)_{NS}$ is completely controllable.

\end{proof}

\begin{remark}
The proofs of Theorems \ref{t2} and \ref{t3} are similar to that of Theorems 4.8 and 4.9 of Section 4 in \cite{ChitourKokkonen}.
\end{remark}

\subsection{The Rolling Problem $\Sigma_{R}$}

From Proposition \ref{p4}, we get the subsequent proposition and corollary whose proofs follow those of Proposition 5.20 and Corollary 5.21 of Section 5 in \cite{ChitourKokkonen}.

\begin{proposition}\label{p21}
Let $\qz \in Q$. Suppose that, for some $X \in \VF(M)$ and a real sequence $(t_{n})_{n=1}^{\infty}$ such that $t_{n} \neq 0$ for all $n$, $\lim_{n \rarrow \infty} t_{n} =0$, we have
\begin{equation}\label{T}
V |_{\Phi_{\lr (X)} (t_{n},q_0)} (\pq) \subset T(\odr (q_0)), \;\;\; \forall n.
\end{equation}
Then $\lns(Y,\hY) |_{q_0} \in T_{q_0} (\odr (q_0))$ for every $Y$ $g$-orthogonal to $X |_{x_0}$ in $T_{x_0} M$ and every $\hY$ $\hg$-orthogonal to $A_0 X |_{x_0}$, $\in A_0 (X |_{x_0})^{\perp}$ in $T_{\hx_0} \hM$. Hence the orbit $\odr (q_0)$ has codimension at most $| \hn - n | +1$ inside $Q$.
\end{proposition}

\begin{corollary}\label{c5}
Suppose there is a point $\qz \in Q$ and $\epsilon > 0$ such that for every $X \in \VF (M)$ with $\Vert X \Vert_{g} < \epsilon$ on $M$, one has
$$
V |_{\Phi_{\lr(X)} (t,q_0)} (\pq) \subset T \odr(q_0), \;\;\; | t | < \epsilon.
$$
Then the orbit $\odr(q_0)$ is open in $Q$.
As a consequence, $(\Sigma)_{R}$ is completely controllable if and only if
\begin{equation}
\forall q \in Q, \;\;\; V |_{q} (\pq) \subset T_{q} \odr (q).
\end{equation}
\end{corollary}

\begin{remark}\label{r5}
We will use in the next corollary the fact that we have $\dr \onq$ is involutive if and only if $Rol_q$ is vanish for all $\q \in Q$, i.e. if and only if $\hR (AX,AY )(AZ) = A(R(X, Y )Z), \text{ for all } X, Y, Z \in T_{x} M$. This is an immediate result from the equality (\ref{S}) and the decomposition of Remark \ref{r1}.
\end{remark}

\begin{corollary}\label{c4}
Assume that $n \leq \hn $. Then the following two cases are equivalent,
\begin{enumerate}
\item[$(i)$]
$\dr$ is involutive,
\item[$(ii)$]
$(M, g)$ and $(\hM , \hg)$ have constant and equal curvature.
\end{enumerate}
Moreover, if $n < \hn$ strictly, then the following two cases are equivalent,
\begin{enumerate}
\item[$(a)$]
$\widehat{\dr}$ is involutive,
\item[$(b)$]
$(M, g)$ and $(\hM , \hg)$ are both flat.
\end{enumerate}
\end{corollary}

\begin{proof}
The proof of $(i)\Leftrightarrow(ii)$ is similar to that of Corollary 5.23 of Section 5 in \cite{ChitourKokkonen}.  We next turn to the proof of $(a)\Rightarrow(b)$.
Assume that $\widehat{\dr}$ is involutive i.e., for every $\hat{q}=(\hx,x;B) \in \hat{Q}, \;\;\; \hX, \hY, \hZ \in T_{\hx} \hM$,
$$
\widehat{Rol}_{\hat{q}}(\hX,\hY) \hZ = B(\hR(\hX,\hY)\hZ) - R(B \hX,B \hY)(B \hZ) = 0.
$$
Thus, we have, for any $X, Y \in T_x M$,
$$
\sigma_{(X,Y)} = g ( R (X,Y)Y,X) = g ( R (B B^{T} X,B B^{T} Y)(B B^{T} Y), X) = g ( B (\hR (B^{T} X, B^{T} Y)(B^{T} Y)), X).
$$
Since $g ( B (\hR (B^{T} X, B^{T} Y)(B^{T} Y)), X)$ = $\hg ( \hR (B^{T} X, B^{T} Y)(B^{T} Y), B^{T} X)$, one deduces that $\sigma_{(X,Y)}$ is equal to $ \hat{\sigma}_{(B^{T} X, B^{T} Y)}$.
Given any $x \in M$, $\hx \in \hM$, $X$, $Y \in T_{x} M $ and $\hX$, $\hY \in T_{\hx} \hM$, choose some vectors $X_3, . . . ,X_{n} \in T_{x} M$ and $\hX_3, . . . , \hX_{\hn} \in T_{\hx} \hM$ such that $X,Y,X_3, . . . ,X_{n}$ and $\hX , \hY , \hX_3, . . . , \hX_{\hn}$ are positively oriented orthonormal frames. We define
$$
B\hX = X, B\hY =Y,\quad
B \hX_{i} = 0; \;\;\;  i = 3, . . . , n,\quad
B \hX_{i} = 0; \;\;\;  i = n+1, . . . , \hn.
$$
Clearly, $\hat{q}=(\hx,x;B) \in \hat{Q}$ and $\sigma_{(X,Y)} = \hat{\sigma}_{(\hX,\hY)}$ for $B^{T}X = \hX $, $B^{T} Y = \hY$. Thus $(M, g)$ and $( \hM , \hg)$ have equal and constant curvature $k \in \mathbb{R}$. We need to show that $k=0$. Choose any $(\hx,x;B) \in \hat{Q}$, since $n < \hn$, choose non-zero vectors $\hX \in \ker B$ and $\hY \in (\ker B)^{\perp}$ and compute
$$
0 = \widehat{Rol}(\hX,\hY)(B) \hX = k (\hg (\hY,\hX) B \hX - \hg (\hX,\hX) B \hY) - R (B \hX, B \hY) (B \hX) = - k \Vert \hX \Vert_{\hg}^2 B \hY.
$$
However $\Vert \hX \Vert_{\hg} \neq 0$ and $B \hY \neq 0$, it follows that $k = 0$.

We now prove that $(b)\Rightarrow(a)$.
In the case where $(M, g)$ and $( \hM , \hg)$ are flat, we have $R = 0$ and $\hR = 0$ so that clearly $\widehat{Rol}(\hX,\hY)(B) \hZ = B(\hR(\hX,\hY)\hZ) - R(B \hX,B \hY)(B \hZ) = 0$ for all $(\hx,x;B) \in \hat{Q}$ and $\hX$, $\hY$, $\hZ \in T_{\hx} \hM$. This proves that $\widehat{\dr}$ is involutive.
\end{proof}

We have another equivalence relation similar to Corollary 5.24 of Section 5 in \cite{ChitourKokkonen}.

\begin{proposition}\label{p17}
Suppose that $(M,g)$ and $(\hM,\hg)$ are complete. The following cases are equivalent:
\begin{enumerate}
\item[(i)]
There exists a $\qz \in Q$ such that $\odr(q_0) $ is an integral manifold of $\dr$.
\item[(ii)]
There exists a $\qz \in Q$ such that,
$$
Rol_q (X,Y) =0, \quad \forall \q \in \odr(q_0), \; X,Y \in T_{x} M.
$$
\item[(iii)]
There is a complete Riemannian manifold $(N,h)$, a Riemannian covering map $F: N \rarrow M$ and a smooth map $G: N \rarrow \hM$ such that
\begin{enumerate}
\item[(1)]
If $n \leq \hn$, $G$ is a Riemannian immersion that maps $h$-geodesics to $\hg$-geodesics.
\item[(2)]
If $n \geq \hn$, $G$ a Riemannian submersion such that the co-kernel distribution $(\ker G_{*})^{\perp} \subset TN$ is involutive and the fibers $G^{-1} (\hx)$, $\hx \in \hM$, are totally geodesic submanifolds of $(N,h)$.
\end{enumerate}

\end{enumerate}
\end{proposition}

\begin{proof}
We will first establish the equivalence $(i)\Longleftrightarrow(ii)$ and to complete the proof, we proceed to show that $(i)\Rightarrow(iii)$ and $(iii)\Rightarrow(ii)$.

We prove $(i)\Rightarrow(ii)$. Notice that the restrictions of vector fields $\lr (X)$, with $X \in \VF(M)$, to the orbit $\odr(q_0)$ are smooth vector fields of that orbit. Thus $[\lr(X),\lr(Y)]$ is tangent to this orbit for any $X$, $Y \in \VF(M)$ and hence (\ref{S}) implies the claim.

We next prove $(ii)\Rightarrow(i)$. From (\ref{S}), it also follows that $\dr |_{\odr(q_0)}$, the restriction of $\dr$ to the manifold $\odr(q_0)$, is involutive. Since the maximal connected integral manifolds of an involutive distribution are exactly its orbits, we get that $ \odr (q_0)$ is an integral manifold of $\dr$.

We now prove $(i)\Rightarrow(iii)$. Let $N:= \odr(q_0)$ and $h:=(\pi_{Q,M} |_{N})^{*} (g)$ i.e. for $\q \in N$ and $X$, $Y \in T_{x} M$, define
$$
h(\lr(X) \onq, \lr(Y) \onq ) = g(X,Y).
$$
If $F:= \pi_{Q,M} |_{N}$ and $G:= \pi_{Q,\hM} |_{N}$, we immediately see that $F$ is a local isometry (note that $\dim (N) =n$). The completeness of $(N,h)$ follows from the completeness of $M$ and $\hM$ with Remark \ref{r9}. Hence $F$ is a surjective Riemannian covering.
Moreover, if $\overline{\Gamma} : [0,1] \rarrow N$ is a $h$-geodesic, it is tangent to $\dr$ and since it projects by $F$ to a $g$-geodesic $\gamma$, it follows again by Remark \ref{r9} that $G \circ \overline{\Gamma} = \hat{\gamma}_{\dr} (\gamma, \overline{\Gamma} (0))$ is a $\hg$-geodesic.
Therefore we have proven that $G$ is a totally geodesic mapping $N\to\hM$.

If $n\leq \hn$, then for $\q \in N$, $X$, $Y \in T_{x} M$, one has
\[
\hg(G_{*} (\lr (X) \onq ), G_{*} (\lr (Y) \onq)) = \hg (AX , AY) = g (X,Y) = h_{*} (\lr (X) \onq , \lr (Y) \onq),
\]
i.e. $G$ is a Riemannian immersion. Item $(1)$ is proved.

If $n \geq \hn$, for $\q \in N$ and $X \in T_{x} M$ such that $\lr(X)\onq \in (\ker G_{*} \onq)^{\perp}$ and $Z \in \ker A$, we have $G_{*} (\lr(Z) \onq) = AZ =0$ i.e. $\lr(Z) \onq  \in  \ker (G_{*} \onq)$ from which $g(X,Z) = h (\lr(X) \onq , \lr(Z) \onq) = 0$ for all $Z \in \ker A$. This shows that $X \in (\ker A)^{\perp}$. Therefore, for all $X$, $Y \in T_{x} M$ such that $\lr(X) \onq $, $\lr(Y) \onq \in (\ker G_{*} \onq)^{\perp}$, we get $\hg (G_{*} (\lr(X) \onq), G_{*} (\lr(Y) \onq)) = h (\lr(X) \onq,\lr(Y) \onq) $ as above. This proves that $G: N \rarrow \hM$ is a Riemannian submersion, which is also totally geodesic.
It then follows from Theorem 3.3 in \cite{Vilms}, that the
fibers of $G$ are totally geodesic submanifolds of $N$ and
that the co-kernel (i.e. horizontal) distribution $(\ker G_*)^\perp$ is involutive. Item $(2)$, and hence the implication $(i)\Rightarrow(iii)$ has been proved.

We next prove $(iii)\Rightarrow(ii)$.
Let $x_0 \in M$ and choose $z_0 \in N$ such that $F(z_0) = x_0$. Define $\hx_0 = G (z_0) \in \hM$ and $A_0 := G_{*} |_{z_0} \circ (F_{*} |_{z_0})^{-1} : T_{x_0} M \rarrow T_{\hx_0} \hM$. The fact that $\qz \in Q$ can be seen as follows: if $(iii) - (1)$ holds, we have
$$
\hg (A_0 X, A_0 Y)  =\hg (G_{*} |_{z_0} ((F_{*} |_{z_0})^{-1} X),G_{*} |_{z_0} ((F_{*} |_{z_0})^{-1} Y))  =  h ((F_{*} |_{z_0})^{-1} X,(F_{*} |_{z_0})^{-1} Y)= g (X,Y),
$$
where we used that $G$ is a Riemannian immersion. If $(iii) - (2)$ holds, take $X$, $Y \in (\ker A_0)^{\perp}$, clearly $(F_{*} |_{z_0})^{-1} X$, $(F_{*} |_{z_0})^{-1} Y \in (\ker G_{*} |_{z_0})^{\perp}$ and hence $\hg (A_0 X, A_0 Y) = g (X,Y)$ because $G$ is a Riemannian submersion.

Let $\gamma : [0,1] \rarrow M$ be a smooth curve with $\gamma (0) = x_0$. Since $F$ is a smooth covering map, there is a unique smooth curve $\Gamma : [0,1] \rarrow N$ with $\gamma = F \circ \Gamma$ and $\Gamma(0)=z_0$. Define $\hgamma = G \circ \Gamma$ and $A(t) = G_{*} |_{\Gamma (t)} \circ (F_{*} |_{\Gamma (t)})^{-1}$, $t \in [0,1]$. As before, it follows that
$q(t)=(\gamma(t),\hgamma(t); A(t)) \in Q$ for all $t \in [0,1]$ and
\begin{equation}\label{U}
\dot{\hgamma} (t) = G_{*} |_{\Gamma (t)} \dot{\Gamma} (t) = A(t) \dot{\gamma} (t).
\end{equation}
According to Theorem 3.3 in \cite{Vilms},
the subcases (1) and (2) mean, respectively,
that $G$ is a totally geodesic map, which is moreover a Riemannian (1) immersion,
(2) submersion.
By Corollary 1.6 in \cite{Vilms}, $G$ is then affine map
i.e. preserves parallel transport.
But $F$, being a Riemannian covering map,
also preserves parallel transport, i.e. is affine.
It follows that $A(t)=G_{*} |_{\Gamma (t)} \circ (F_{*} |_{\Gamma (t)})^{-1}$
also preserves parallel transport, which combined with \eqref{U}
means that $A(t)$ is the rolling curve along $\gamma$
with $A(0)=A_0$.

Since the affinity of $F$ (resp. $G$)
simply means that 
$$
\nabla_{F_*\overline{X}} (F_*(\overline{Y}))=F_*(\nabla^h_{\overline{X}} \overline{Y})
(\hbox{ resp. } \hnabla_{G_*\overline{X}} (G_*(\overline{Y}))=G_*(\nabla^h_{\overline{X}} \overline{Y})),
$$
for all vector fields $\overline{X}, \overline{Y}$ on $N$,
we easily see that
\[
& R(F_*\overline{X}, F_*\overline{Y})F_*\overline{Z}=F_*(R^h(\overline{X}, \overline{Y})\overline{Z}) \\
& \hR(G_*\overline{X}, G_*\overline{Y})G_*\overline{Z}=G_*(R^h(\overline{X}, \overline{Y})\overline{Z}),
\]
for all vector fields $\overline{X}, \overline{Y}, \overline{Z}$ on $N$.
It thus follows that for all vector fields $X,Y,Z$ on $M$,
\[
A(t)(R(X,Y)Z)=&A(t)(R(F_*\overline{X}, F_*\overline{Y})F_*\overline{Z})
=A(t)(F_*(R^h(\overline{X}, \overline{Y})\overline{Z})) \\
=&G_*|_{\Gamma(t)}(R^h(\overline{X}, \overline{Y})\overline{Z}))
=\hR(G_*|_{\Gamma(t)}\overline{X}, G_*|_{\Gamma(t)}\overline{Y})G_*|_{\Gamma(t)}\overline{Z} \\
=&\hR(A(t)X,A(t)Y)(A(t)Z),
\]
where $\overline{X}, \overline{Y}, \overline{Z}$ are any (local) $F$-lifts of $X,Y,Z$ on $N$.
This proves that
\begin{equation}\label{W}
\Rol_{q(t)} = 0.
\end{equation}
Thus we have shown that $t \mapsto (\gamma (t) , \hgamma (t); A(t))$ is the unique rolling curve along $\gamma$ starting at $\qz$ and defined on $[0,1]$ and therefore curves of $Q$ formed in this manner fill up the orbit $\odr (q_0)$. Moreover, by Eq. (\ref{W}) we have shown also that $\Rol$ vanishes on $\odr (q_0)$.
\end{proof}

\begin{remark}
As pointed out in the course of the above proof,
according to \cite{Vilms} the subcases (1)-(2) of (iii) in the previous proposition can be replaced by simply
saying that $G$ is a totally geodesic map which is a Riemannian (1) immersion, (2) submersion, respectively.
\end{remark}

The next proposition is a sufficient condition of non-controllability for the rolling system $\Sigma_{(R)}$ when $n < \hn$.
\begin{proposition}
Let $M$, $\hM$ be two Riemannian manifolds of dimensions $n$, $\hn$ with $n < \hn$. Assume that there exists a complete totally geodesic submanifold $\hN$ of $\hM$ of dimension $m$ such that $n \leq m<\hn $. Then, the rolling system $\Sigma_{(R)}$ of $Q (M, \hM)$ is not completely controllable.
\end{proposition}

\begin{proof}

Since $n \leq m$, we can find $q_0 = (x_0, \hx_0 ; A_0) \in Q$ such that $\hx_0 \in \hN$ and $\IM (A_0) \subset T_{\hx_0} \hN$.
We proceed to prove that $\pi_{Q,\hM} (\odr (q_0)) \subset \hN$. To this end, we will first prove that for every geodesic curve $\gamma$ on $M$ starting at any point $q=(x,\hx;A)$, with $x\in M$,
$\hx\in\hN$ and $\IM (A) \subset T_{\hx} \hN$, the resulting geodesic curve
$\hgamma_{\dr}:=\hgamma_{\dr} (\gamma , q)= \pi_{Q, \hM} (q_{\dr} (\gamma, q))$ stays in $\hN$
and that if $q_{\dr} (\gamma, q)=(\gamma,\hgamma;A_{\dr} (\gamma, q))$, then
$\IM A_{\dr} (\gamma, q)(\cdot)\subset T_{\hgamma(\cdot)} \hN$.

Once this proved, it is clearly obvious that the previous statement extends verbatim to the case where $\gamma$ is any broken geodesic curve. By a standard density argument, we conclude that the above statement is again true for any absolutely continuous curve $\gamma$ on $M$. We then prove the claim.

Let then consider a point $q=(x,\hx;A)$, with $x\in M$,
$\hx\in\hN$ and $\IM (A) \subset T_{\hx} \hN$ and a geodesic curve $\gamma: [0,1] \rarrow M$ starting at $x\in M$. Then, $q_{\dr} (\gamma , q)$ is a geodesic curve and so that $\hgamma_{\dr} (\gamma , q)$ is a geodesic curve on $\hM$ and for all $t \in [0, 1]$, we have,
\begin{align*}
\dot{\hgamma}_{\dr} (t) =& \dot{\hgamma}_{\dr} (\gamma , q) (t) =  A_{\dr} (\gamma , q) (t) \dot{\gamma} (t)
=  ( P_0^t (\hgamma_{\dr}) \circ A \circ P_t^0 (\gamma) ) \dot{\gamma} (t)\\
= & P_0^t (\hgamma_{\dr}) ( A\dot{\gamma}(0) ).
\end{align*}
By assumption $\IM (A) \subset T_{\hx} \hN$,
and therefore $A\dot{\gamma}(0) \in T_{\hx}\hN$,
which implies that $\dot{\hgamma}_{\dr} (0)\in T_{\hx}\hN$.
Since $\hN$ is a complete totally geodesic submanifold of $\hM$,
we therefore have that the geodesic $\hgamma_{\dr}(t)$ stays in $N$
for all $t\in [0,1]$.

Using the same reasoning, for a given $t\in [0,1]$,
if $X\in T_{\gamma(t)} M$, we have
$A(P_t^0(\gamma)X)\in T_{\hx} \hN$,
and hence, since $\hN$ is totally geodesic,
$A_{\dr}(\gamma,q)(t)X\in T_{\hgamma(t)} \hN$.
This combined with the fact that $A_{\dr}(\gamma,q)(t)$ preserves
the inner product $\hg$ of $\hM$, and therefore that induced on $\hN$,
means that $q_{\dr}(\gamma,q)(t)\in Q(M,\hN)$ for all $t\in [0,1]$,
which completes the proof.
\end{proof}

Since Riemannian manifolds $(\hM,\hg)$ of constant curvature contain complete totally geodesic submanifolds of any lower dimension, we get the following non-controllability result as consequence of the previous proposition.
\begin{corollary}\label{c10}
Consider a Riemannian manifold $(M,g)$ of dimension $n$ and a Riemannian manifold $(\hM,\hg)$ of constant curvature and of dimension $\hn>n$. Then the rolling problem of $(M,g)$ onto $(\hM,\hg)$  without spinning nor slipping  is not controllable.
\end{corollary}

\section{Appendix}\label{app0}

In this section we briefly show how one writes
the control system $\Sigma_{(R)}$ in local orthonormal frames.

Let $(F_i)_{1\leq i\leq n}$ and $(\hat{F}_j)_{1\leq j\leq \hn}$ be local oriented orthonormal frames on $M$ and $\hat{M}$ respectively and let $q_0=(x_0,\hat{x}_0;A_0)\in Q$
such that $x_0$, $\hat{x}_0$ belong to the domains of definition $V$ and $\hat{V}$ of the frames.
Let $q(t)=(\gamma(t),\hat{\gamma}(t);A(t))$, $t\in [0,1]$, be a curve in $Q$ so that
$\gamma\subset V$ and $\hat{\gamma}\subset \hat{V}$.
For every $t\in [0,1]$, define the unique element $\mc{R}(t)$ in $\SO(n,\hn)$ verifying
\[
\big(A(t)F_1|_{\gamma(t)},\dots,A(t)F_n|_{\gamma(t)}\big)
=\big(\hat{F}_1|_{\hat{\gamma}(t)},\dots,\hat{F}_{\hn}|_{\hat{\gamma}(t)}\big)\mc{R}(t)
\]

Define Christoffel symbols $\Gamma\in T^*_x M\otimes \mathfrak{so}(n)$ and $\hat{\Gamma}\in T^*_{\hat{x}} \hat{M}\otimes \mathfrak{so}(\hn)$ by
\[
\Gamma(X)_i^l=g(\nabla_X F_i, F_l),
\quad
\hat{\Gamma}(\hat{X})_j^k=\hat{g}(\hat{\nabla}_{\hat{X}} \hat{F}_j, \hat{F}_k),
\]
with $1\leq i,k\leq n$, $1\leq j,k\leq \hn$ and $X\in T_x M$, $\hX\in T_{\hx}\hM$.

There are unique measurable functions $u^i:[0,1]\to\R$, $1\leq i\leq n$, such that, for a.e. $t\in [0,1]$,
\[
\dot{\gamma}(t)=&\big(F_1|_{\gamma(t)},\dots,F_n|_{\gamma(t)}\big)\qmatrix{u^1(t)\cr \vdots \cr u^n(t)}.
\]
As one can easily verify,
the conditions of no-slip \eqref{CN} and no-spin \eqref{CM}
translate for $(\hat{\gamma}(t),\mc{R}(t))\in \hat{M}\times\SO(n)$
precisely to
\[
\textrm{(no-slip)}\quad & \dot{\hat{\gamma}}(t)=\big(\hat{F}_1|_{\hat{\gamma}(t)},\dots,\hat{F}_{\hn}|_{\hat{\gamma}(t)}\big)\mc{R}(t)\qmatrix{u^1(t)\cr \vdots \cr u^n(t)}, \\
\textrm{(no-spin)}\quad & \dot{\mc{R}}(t)=\mc{R}(t)\Gamma(\dot{\gamma}(t))-\hat{\Gamma}(\dot{\hat{\gamma}}(t))\mc{R}(t),
\]
for a.e. $t\in [0,1]$. Moreover, the latter no-spin condition can also be written as
\[
\dot{\mc{R}}(t)=\sum_{i=1}^n u^i(t)\Big(\mc{R}(t)\Gamma(F_i|_{\gamma(t)})-\sum_{j=1}^{\hn} \mc{R}_{ji}(t)\hat{\Gamma}(\hat{F}_j|_{\hat{\gamma}(t)})\mc{R}(t)\Big),
\]
for a.e. $t\in [0,1]$, where $\mc{R}_{ji}(t)$ is the element at $j$-th row, $i$-th column of $\mc{R}(t)$. From this local form, one clearly sees that the rolling system $\Sigma_R$
is a driftless control affine system (see \cite{agrachev04,jurd} for more details on control systems).


\begin{thebibliography}{99}
\bibitem{agrachev99} A. Agrachev and Y.  Sachkov, \emph{An intrinsic Approach to the control of Rolling Bodies}. Proceedings of the CDC, Phoenix, 1999, pp. 431 - 435, vol. I.
\bibitem{agrachev04} Agrachev, A., Sachkov, Y., \emph{Control Theory from the Geometric Viewpoint},
Encyclopaedia of Mathematical Sciences, 87. Control Theory and Optimization, II. Springer-Verlag, Berlin, 2004.
   
\bibitem{ACL} Alouges, F., Chitour Y., Long, R., \emph{A motion planning algorithm for the rolling-body problem}, Robotics, IEEE Transactions on 26 (5), 827-836, 2010.

\bibitem{bryant-hsu} Bryant, R. and Hsu, L., \emph{Rigidity of integral curves of rank 2 distributions}, Invent. Math. 114 (1993), no. 2, 435--461.
\bibitem{chelouah01} Chelouah, A. and Chitour, Y., \emph{On the controllability and trajectories generation of rolling surfaces}. Forum Math. 15 (2003) 727-758.

   
\bibitem{CGK13:A} Y. Chitour, M. Godoy Molina, and P. Kokkonen, 
\emph{The Rolling Problem: Overview and Challenges}. arXiv: 1301.6370, 2013.

\bibitem{CGK13:B} Y. Chitour, M. Godoy Molina, and P. Kokkonen,  \emph{Symmetries of the Rolling Model}. arXiv: 1301.2579, 2013.

\bibitem{ChitourKokkonen} Y. Chitour and P. Kokkonen, \emph{Rolling Manifolds: Intrinsic Formulation and Controllability}. Preprint, arXiv:1011.2925v2, 2011.

\bibitem{ChitourKokkonen1} Y. Chitour and P. Kokkonen, \emph{Rolling Manifolds and Controllability: the 3D case}. Submitted, 2012.

\bibitem{jurd} Jurdjevic, V. \emph{Geometric control theory}, Cambridge Studies in Advanced Mathematics, 52. Cambridge University Press, Cambridge, 1997.

\bibitem{KobayashiNomizu} S. Kobayashi and K. Nomizu, \emph{Foundations of Differential Geometry}. Vol. I, Wiley-Interscience, 1996.

\bibitem{kokkonen12:A} P. Kokkonen, \emph{A characterization of isometries between Riemannian manifolds by using development along geodesic triangles}, Archivum Mathematicum, Vol. 48 (2012), No. 3, 207--231.

\bibitem{kokkonen12:B} P. Kokkonen, \emph{\'Etude du mod\`ele des vari\'et\'es roulantes et de sa commandabilit\'e}, PhD thesis 2012.
(http://tel.archives-ouvertes.fr/tel-00764158)

\bibitem{marigo-bicchi} Marigo, A. and Bicchi A., \emph{Rolling bodies with regular surface: controllability theory and applications}, IEEE Trans. Automat. Control 45 (2000), no. 9, 1586--1599.

\bibitem{mar-bic2} Marigo, A. and Bicchi A., \emph{Planning motions of polyhedral parts by rolling},
Algorithmic foundations of robotics. Algorithmica 26 (2000), no. 3-4, 560--576. 

\bibitem{norway} Molina M., Grong E., Markina I., Leite F., \emph{An intrinsic formulation of the problem of rolling manifolds}, Journal of dynamical and control systems 18 (2), 181-214, 2012.
\bibitem{Sakai} T. Sakai, \emph{Riemannian Geometry}. Translations of Mathematical Monographs, 149. American Mathematical Society, Providence, RI, 1996.

\bibitem{sharpe97} Sharpe, R.W., \emph{Differential Geometry: Cartan's Generalization of Klein's Erlangen Program},
Graduate Texts in Mathematics, 166. Springer-Verlag, New York, 1997.

\bibitem{spivakII99} Spivak, M., \emph{A Comprehensive Introduction to Differential Geometry}, Vol. 2, Publish or Perish, 3rd edition, 1999.

\bibitem{Vilms} J. Vilms, \emph{Totally Geodesic Maps}. J. Differential Geometry, 4 (1970) 73-79.

\end{thebibliography}
\end{document}